\documentclass[a4paper,abstracton]{scrartcl}
\usepackage{amsfonts}
\usepackage{amssymb}
\usepackage{amsmath}
\usepackage{amsthm}

\usepackage[ruled,vlined,linesnumbered]{algorithm2e}

\usepackage{natbib}
\usepackage{bm}
\usepackage{graphicx}

\usepackage{hyperref}
\usepackage{xcolor}
\definecolor{darkgreen}{RGB}{0,127,173}
\definecolor{darkblue}{RGB}{0,0,200}
\hypersetup{
	colorlinks,
        linkcolor={darkblue},
        citecolor={darkgreen},
        urlcolor={darkblue}
}

\usepackage{caption}
\usepackage{subcaption}

\usepackage{rotating}
\usepackage{multirow}

\usepackage{chngpage}
\usepackage{booktabs}
\usepackage{verbatim}

\newtheorem{theorem}{Theorem}

\newtheorem{cor}[theorem]{Corollary}

\usepackage{authblk}

\allowdisplaybreaks

\begin{document}

\newcommand{\X}{{\mathcal{X}}}
\newcommand{\cU}{{\mathcal{U}}}
\newcommand{\cI}{{\mathcal{I}}}
\newcommand{\cC}{{\mathcal{C}}}
\newcommand{\N}{{\mathcal{N}}}

\newcommand{\cM}{{\mathcal{M}}}
\newcommand{\cK}{{\mathcal{K}}}
\newcommand{\cE}{{\mathcal{E}}}
\newcommand{\cV}{{\mathcal{V}}}
\newcommand{\cW}{{\mathcal{W}}}

\newcommand{\change}{\color{blue}}   
\newcommand{\changem}{\color{red}}   
\title{Recoverable Robust Single Machine Scheduling with Polyhedral Uncertainty}

 \author[1]{Matthew Bold\footnote{corresponding author}}
 \author[2]{Marc Goerigk}
 
 \affil[1]{STOR-i Centre for Doctoral Training, Lancaster University, United Kingdom\\ m.bold1@lancaster.ac.uk}
 \affil[2]{Network and Data Science Management, University of Siegen, Germany\\ marc.goerigk@uni-siegen.de}
    
\date{}

\maketitle

\begin{abstract}
This paper considers a recoverable robust single-machine scheduling problem under polyhedral uncertainty with the objective of minimising the total flow time. In this setting, a decision-maker must determine a first-stage schedule subject to the uncertain job processing times. Then following the realisation of these processing times, they have the option to swap the positions of up to $\Delta$ disjoint pairs of jobs to obtain a second-stage schedule. 

We first formulate this scheduling problem using a general recoverable robust framework, before we examine the incremental subproblem in further detail. We prove a general result for max-weight matching problems, showing that for edge weights of a specific form, the matching polytope can be fully characterised by polynomially many constraints. We use this result to derive a matching-based compact formulation for the full problem. Further analysis of the incremental problem leads to an additional assignment-based compact formulation. Computational results on budgeted uncertainty sets compare the relative strengths of the three compact models we propose.
\end{abstract}



\section{Introduction}

We consider a scheduling problem where $n$ jobs must be scheduled on a single machine without preemption, such that the total flow time, i.e. the sum of completion times, is minimised. This problem is denoted as $1||\sum C_i$ under the $\alpha|\beta|\gamma$ scheduling problem notation introduced by \cite{graham1979optimization}. In practice, job processing times are often subject to uncertainty, and when this is the case it is important to find robust solutions that account for this uncertainty. In this paper, we propose a recoverable robust approach \citep{liebchen2009concept} to this uncertain single machine scheduling problem. In this recoverable robust setting, we determine a full solution in a first-stage, before an adversarial player chooses a worst-case scenario of processing times from an uncertainty set, and then in response to this, we allow the first-stage solution to be adjusted in a limited way.

The deterministic single machine scheduling problem (SMSP) is one of the simplest and most studied scheduling problems, and can be solved easily in $O(n\log n)$ time by ordering the jobs according to non-decreasing processing times, i.e. by using the shortest processing time (SPT) rule. However, despite the simplicity of the nominal problem, the robust problem has been shown to be NP-hard for even the most basic uncertainty sets \citep{daniels1995robust}. 

In fact, the majority of research to date regarding robust single machine scheduling has been concerned with the presentation of complexity results for a number of different SMSPs. First discussed by \cite{daniels1995robust}, \cite{kouvelis1997robust} and \cite{yang2002robust}, these papers study the problem with the total flow time objective, and show that it is NP-hard even in the case of two discrete scenarios, for min-max, regret and relative regret robustness. Robust single machine scheduling for discrete uncertain scenarios has been examined extensively. \cite{aloulou2008complexity} present algorithmic and complexity results for a number of different SMSPs under min-max robustness. \cite{aissi2011minimizing} show that the problem of minimising the number of late jobs in the worst-case scenario, where processing times are known, but due dates are uncertain is NP-hard. \cite{zhao2010family} consider the objective of minimising the weighted sum of completion times in the worst-case scenario, and propose a cutting-plane algorithm to solve the problem. \cite{mastrolilli2013single} study this same problem and show that no polynomial-time approximation scheme exist for the unweighted version. \cite{kasperski2016single} apply the ordered weighted averaging (OWA) criterion, of which classical robustness is a special case, to a number of different SMSPs under discrete uncertainty. The consideration of SMSPs under novel optimality criteria has been continued most recently by \cite{kasperski2019risk}, where a number of complexity results are presented for the SMSP with the value at risk (VaR) and conditional value at risk (CVaR) criteria.

Robust single machine scheduling in the context of interval uncertainty has also received considerable attention. \cite{daniels1995robust} address interval uncertainty, and describe some dominance relations between the jobs in an optimal schedule based on their processing time intervals. \cite{kasperski2005minimizing} considers an SMSP with precedence constraints, and where the regret of the maximum lateness of a job is to be minimised. A polynomial-time algorithm is presented. \cite{lebedev2006complexity} show that the SMSP with the total flow time objective is NP-hard in the case of regret robustness. \cite{montemanni2007mixed} present a mixed-integer program (MIP) for this same problem, and use it to solve instances involving up to 45 jobs. \cite{kasperski20082} also consider this problem, and show that it is 2-approximable when the corresponding deterministic problem is polynomially solvable. \cite{lu2012robust} present an SMSP with uncertain job processing and setup times, show this problem is NP-hard, and design a simulated annealing-based algorithm to solve larger instances. \cite{chang2017distributionally} apply distributional robustness to an SMSP, and make use of information about the mean and covariance of the job processing times to minimise the worst-case CVaR. Most recently, \cite{fridman2020minimizing} consider an SMSP with uncertain job processing times and develop polynomial algorithms for solving the min-max regret problem under certain classes of cost functions. For a survey of robust single-machine scheduling in the context of both discrete and interval uncertainty, see \cite{kasperski2014minmax}. 

A criticism of classical robustness is that the solutions it provides are overly conservative and hedge against extreme worst-case scenarios that are very unlikely to occur in practice. To reduce the level of conservatism, a restriction to interval uncertainty was introduced by \cite{bertsimas2004price}, known as budgeted uncertainty, in which the number of jobs that can simultaneously achieve their worst-case processing times is restricted. Budgeted uncertainty is a special case of the general compact polyhedral uncertainty that is considered in this paper. Robust single machine scheduling under budgeted uncertainty was first considered by \cite{lu2014robust}, who present an MIP and heuristic to solve the problem. Following this, \cite{tadayon2015algorithms} study different versions of the min-max robust SMSP under three different uncertainty sets, including a budgeted uncertainty set. Recently, \cite{bougeret2019robust} present complexity results and approximation algorithms for a number of different min-max robust scheduling problems under budgeted uncertainty.


To the best of our knowledge, this paper is the first to solve a single-machine scheduling problem in a recoverable robust setting. However, recoverable robustness has had recent application to a number of closely related matching, assignment and scheduling problems. \cite{fischer2020investigation} consider a recoverable robust assignment problem, in which two perfect matchings of minimum costs must be chosen, subject to these matchings having at least $k$ edges in common. If the cost of the second matching is evaluated in the worst-case scenario, we arrive in the setting of recoverable robustness with interval uncertainty. Hardness results are presented, and a polynomial-time algorithm is developed for the restricted case in which one cost function is Monge. Recently, \cite{bold2021recoverable} also considered recoverable robust scheduling problems under interval uncertainty, deriving a 2-approximation algorithm for their setting.

Regarding project scheduling, \cite{bendotti2019anchor} introduce the so-called anchor-robust project scheduling problem in which a baseline schedule is designed under the problem uncertainty, such that the largest possible subset of jobs have their starting times unchanged following the realisation of the activity processing times. This problem is shown to be NP-hard even for budgeted uncertainty. In a series of papers \cite{bruni2017adjustable,bruni2018computational,bold2020compact}, a two-stage resource-constrained project scheduling problem with budgeted uncertainty is introduced and solved.

The contributions of this paper are as follows. In Section~\ref{sec:definition} we formally define the recoverable robust scheduling problem that we consider in this paper. In Section~\ref{sec:general} we present a general result that enables the construction of compact formulations for a wide range of recoverable robust problems, and apply this in the context the scheduling problem at hand. We then analyse the stages of the recoverable robust scheduling problem in detail and show that the incremental problem can be solved using a simple linear programming formulation in Section~\ref{sec:complexity}. To this end, we prove a general result for max-weight matching problems, arguing that odd-cycle constraints are not required in problems with weights of a specific form. This formulation of the incremental problem then leads to an alternative matching-based compact problem formulation. Additionally, we transfer the matching result to an assignment-based formulation for the incremental problem, which results in a third compact model. In Section~\ref{section:computational_experiments}, computational experiments are presented, showing the benefits of a recourse action, the effects of the uncertainty on the model, and the strength of the assignment-based formulation. Finally, some concluding remarks and potential directions for future research are given in Section~\ref{section:conclusions}.

\section{Problem definition}
\label{sec:definition}

We consider a single machine scheduling problem with the objective of minimising the sum of completion times. Given a set of jobs $\mathcal{N}=\{1,\dots,n\}$ with processing times $\bm{p}=(p_1,\dots,p_n)$, we aim to find a schedule, i.e. an ordering of the jobs $i\in \mathcal{N}$, that minimises the sum of completion times. This nominal problem is denoted by $1||\sum C_i$ under the $\alpha|\beta|\gamma$ scheduling problem notation introduced by \cite{graham1979optimization}. Recall that this problem is easy to solve; the shortest processing time (SPT) rule of sorting jobs by non-decreasing processing times results in an optimal schedule. This problem can be modelled as the following assignment problem with non-general costs:
\begin{align}
	\min\, &\sum_{i\in \mathcal{N}}\sum_{j\in \N} p_i(n+1-j)x_{ij}\\
\text{s.t. } &\sum_{i\in \N} x_{ij} = 1 & \forall j \in \N \label{assignment1}\\
	     & \sum_{j\in \N} x_{ij} = 1 & \forall i \in \N \label{assignment2}\\
	     & x_{ij}\in\{0,1\} & \forall i,\,j \in \N,
\end{align}
where $x_{ij}=1$ if job $i$ is scheduled in position $j$, and $x_{ij}=0$ otherwise.

We assume the job processing times $p_i,\,i\in \N$ are uncertain, but are known to lie within a given uncertainty set $\cU$. In this paper, we consider a general polyhedral uncertainty set given by
	$$\mathcal{U}=\left\{\bm{p}\in \mathbb{R}^n_+:\bm{A}\bm{p}\leq \bm{b}\right\},$$
where $\bm{A}\in\mathbb{R}^{M\times n}$ and $\bm{b}\in\mathbb{R}^M$, consisting of $M$ linear constraints $a_{m1}p_1+a_{m2}p_2+\dots+a_{mn}p_n\leq b_m$ for $m\in\mathcal{M}=\{1,\dots,M\}$ on the set of possible processing times $\bm{p}$. Throughout this paper, we assume $\cU$ to be compact. It is also possible to include auxiliary variables in the definition of $\cU$; for ease of presentation, such variables have been omitted.

We consider this uncertain single machine scheduling problem in the context of a two-stage decision process, where, having decided on a first-stage schedule $\bm{x}$ under the problem uncertainty, the decision-maker is given the opportunity to react to the realisation of the uncertain data by choosing up to $\Delta$ distinct pairs of jobs and swapping their positions, to obtain a second-stage schedule $\bm{y}$.

This recoverable robust problem can be written as follows:
\begin{align*}
	&\min_{\bm{x}\in\X} \max_{\bm{p}\in\cU} \min_{\bm{y}\in \X(\bm{x})} \sum_{i\in \N} \sum_{j\in \N} p_i(n+1-j)y_{ij},&\textnormal{(RRS)} 
\end{align*}
where $\X=\{\bm{x}\in\{0,1\}^{n\times n}: \eqref{assignment1}, \eqref{assignment2}\}$ is the set of feasible schedules, and $\X(\bm{x})\subseteq\X$ is the set of feasible second-stage assignments given $\bm{x}$. That is,
$$\X(\bm{x}) = \{ \bm{y} \in \X : d(\bm{x},\bm{y}) \le \Delta \},$$
where $d(\bm{x},\bm{y})$ is some measure of the distance between the first and second-stage schedules. 

In this paper, we restrict our attention to the case where the recourse action consist of disjoint pairwise swaps to the first-stage positions of the jobs. This choice of recourse can be motivated with an example. Consider a nuclear storage silo full of ageing containers of untreated nuclear waste that each must undergo a reprocessing procedure to treat the waste and place it into new, safe long-term storage containers. This process is scheduled in advance of its execution and its management is allocated among a number of different store managers. For many of the old waste containers, it is unclear which grade of waste they contain and therefore how extensive their reprocessing procedure will be. After the creation of the preliminary schedule, the nuclear material in each container is examined in detail and the reprocessing time estimates are improved. Having obtained these improved estimates, the reprocessing schedule can be updated. If the tasks that a store manager is responsible for change in the updated schedule, they must meet with the manager that was previously responsible for those tasks in order to be briefed about their technical requirements. To simplify the arrangements of such meetings so that store managers only have to swap briefings with a single other store manager, we restrict schedule updates to disjoint pairwise swaps of waste processing jobs.

In addition to this motivation, the restriction to disjoint pairwise swaps improves the tractability of the problem and leads to the results that we present and analyse in this paper.

Hence, in this case we define $d(\bm{x},\bm{y})$ to be the minimum number of pairwise distinct swaps required to transform $\bm{x}$ into $\bm{y}$, if this number exists; otherwise, we set it to $\infty$. Observe that the value $d(\bm{x},\bm{y})$ can be calculated using the following approach. Let $\mathcal{E}_{\bm{x}}$ be the edges chosen by $\bm{x}$ in the corresponding bipartite graph, oriented towards the right, and let $\mathcal{E}_{\bm{y}}$ be the edges chosen by $\bm{y}$, oriented towards the left, i.e. $(j,i)\in \mathcal{E}_{\bm{y}}$ corresponds to assigning job $i$ to position $j$. If and only if the edges $\mathcal{E}_{\bm{x}} \cup \mathcal{E}_{\bm{y}}$ decompose into 2-cycles and 4-cycles, we have $d(\bm{x},\bm{y}) < \infty$, in which case $d(\bm{x},\bm{y})$ is equal to the number of 4-cycles. This is because a 2-cycle corresponds to a job with an unchanged position, whilst 4-cycles represent a swap of positions of two jobs. An example is given in Figure \ref{fig:distance_ex}.

\begin{figure}[h]
	\centering
	\includegraphics[scale=0.65]{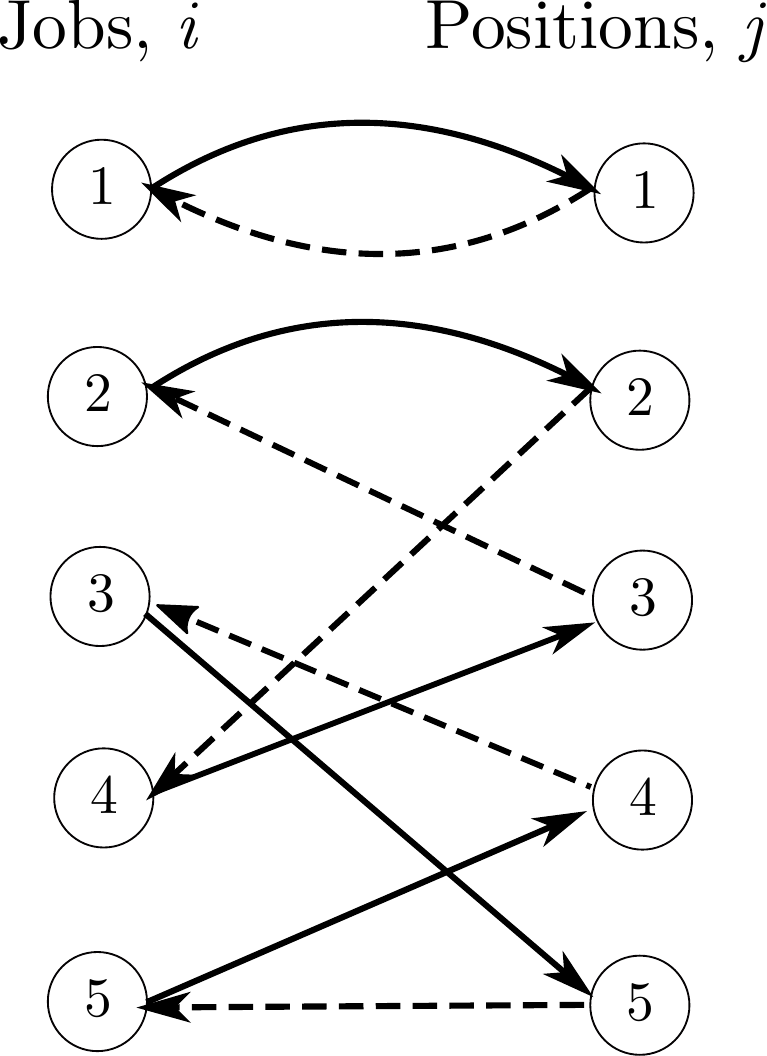}
	\caption{An example first and second-stage solution. The first-stage assignment is given by the solid arcs oriented towards the right, and corresponds to the schedule (1,2,4,5,3). The second-stage assignment is given by the dashed arcs oriented towards the left, and corresponds to the schedule (1,4,2,3,5). There are two 4-cycles corresponding to the switching of positions of jobs 2 and 4, and 3 and 5. Hence $d(\bm{x},\bm{y})=2$.}
	\label{fig:distance_ex}
\end{figure}

We define the adversarial and incremental problems of (RRS) as follows. Given both a first-stage solution $\bm{x}\in\X$ and a scenario $\bm{p}\in\cU$, the incremental problem consists of finding the best possible second-stage solution $\bm{y}\in \X(\bm{x})$. That is,
\[ \text{Inc}(\bm{x},\bm{p}) = \min_{\bm{y}\in\X(\bm{x})} \sum_{i\in \N} \sum_{j\in \N} p_i(n+1-j)y_{ij}. \]
The adversarial problem is to find a worst-case scenario $\bm{p}\in\cU$ for a given first-stage schedule $\bm{x}\in \X$. That is,
\[ \text{Adv}(\bm{x}) = \max_{\bm{p}\in\cU} \min_{\bm{y}\in\X(\bm{x})} \sum_{i\in \N} \sum_{j\in \N} p_i(n+1-j)y_{ij} = \max_{\bm{p}\in\cU} \text{Inc}(\bm{x},\bm{p}). \]

Observe that for the case of general polyhedral uncertainty that we consider here, (RSS) is NP-hard. To see this, suppose that $\Delta=0$, i.e. there is no recovery option and the second-stage variables are fixed to the corresponding first-stage values. Then the problem reduces to a standard robust single machine scheduling problem of the form
$$\min_{\bm{x} \in \X}\max_{\bm{p}\in \mathcal{U}}\sum_{i\in \N}\sum_{j\in \N}p_i(n+1-j)x_{ij}.$$
Note that a general polyhedral uncertainty set can be used to construct a problem involving only two discrete scenarios. This can be done by simply defining the uncertainty set to be the linear combination of the two discrete points, since the worst-case scenario must lie at a vertex of the polyhedron, i.e. at one of the two discrete scenarios. Since the robust scheduling problem with two scenarios is already NP-hard (see \cite{kouvelis1997robust}), this hardness result also extends to our setting.

Finally, note that in problem (RRS) we aim to minimise the worst-case costs of the resulting recovery solutions. If the first-stage costs are also relevant, all the results presented in this paper can be adjusted trivially.

\section{A general model for recoverable robustness}
\label{sec:general}

In this section we present a general model for recoverable robust optimisation problems, and apply this method to the uncertain single machine scheduling problem (RRS). Our approach is to determine a first-stage solution $\bm{x}\in\X$ as well as a finite set of candidate recovery solutions $\bm{y}^1,\ldots,\bm{y}^K\in\X(\bm{x})$. 

The following result shows that $K=n^2+1$ is sufficient to guarantee that this approach provides an exact solution to the problem.

\begin{theorem}
Let a recoverable robust problem of the form
\[ \min_{\bm{x}\in\X} \max_{\bm{c}\in\cU} \min_{\bm{y}\in\X(\bm{x})} f(\bm{y},\bm{c}) \]
be given, where $\X,\X(\bm{x})\subseteq\{0,1\}^n$, $\cU$ is a compact convex set, $f$ is linear in $\bm{y}$, and concave in $\bm{c}$. Then this problem is equivalent to
\[ \min_{\bm{x}\in\X,\atop \bm{y}^{(1)},\ldots,\bm{y}^{(n+1)}\in\X(\bm{x})} \max_{\bm{c}\in\cU} \min_{i=1,\ldots,n+1} f(\bm{y}^{(i)},\bm{c}).\]
\end{theorem}
\begin{proof}
	The idea of the proof is similar to models developed for $K$-adaptability (see \citet[Theorem~1]{hanasusanto2015k} and \citet[Corollary~1]{buchheim2017min}). Recall both Carath\'eodory's theorem and the mini\-max theorem. Carath\'eodory's theorem states that any point $x\in \mathbb{R}^n$ lying in $conv(X)$ can be written as a convex combination of $n+1$ points from $X$. The mini\-max theorem states that if $X$ and $Y$ are two compact, convex sets, and $f:X\times Y\rightarrow\mathbb{R}$ is a continuous compact-concave function (i.e. $f(\cdot,y)$ is concave for fixed values of $y$ and $f(x,\cdot)$ is convex for fixed values of $x$), then $$\max_{x}\min_y f(x,y)=\min_y \max_x f(x,y).$$ We make use of both of these results in the following:
\begin{align*}
 &\min_{\bm{x}\in\X} \max_{\bm{c}\in\cU} \min_{\bm{y}\in\X(\bm{x})} f(\bm{y},\bm{c}) \\
 =& \min_{\bm{x}\in\X} \max_{\bm{c}\in\cU} \min_{\bm{y}\in conv(\X(\bm{x}))} f(\bm{y},\bm{c}) \\
 =& \min_{\bm{x}\in\X} \min_{\bm{y}\in conv(\X(\bm{x}))} \max_{\bm{c}\in\cU}  f(\bm{y},\bm{c})& \textnormal{(by the minimax theorem)}\\
 =& \min_{\bm{x}\in\X} \min_{\bm{y}^{(1)},\ldots,\bm{y}^{(n+1)}\in\X(\bm{x})}
 \min_{\lambda^1,\ldots,\lambda^{n+1}\ge0 \atop \sum_{i=1}^{n+1}\lambda^i = 1} \max_{\bm{c}\in\cU}f( \sum_{i=1}^{n+1}\lambda^i\bm{y}^{(i)},\bm{c}) & \textnormal{(by Carath\'eodory's theorem)}\\
 =& \min_{\bm{x}\in\X} \min_{\bm{y}^{(1)},\ldots,\bm{y}^{(n+1)}\in\X(\bm{x})}
  \max_{\bm{c}\in\cU} \min_{\lambda^1,\ldots,\lambda^{n+1}\ge0 \atop \sum_{i=1}^{n+1} \lambda^i = 1}  f(\sum_{i=1}^{n+1}\lambda^i\bm{y}^{(i)},\bm{c}) & \textnormal{(by the minimax theorem)}\\
 =& \min_{\bm{x}\in\X} \min_{\bm{y}^{(1)},\ldots,\bm{y}^{(n+1)}\in\X(\bm{x})}
  \max_{\bm{c}\in\cU} \min_{\lambda^1,\ldots,\lambda^{n+1}\ge0 \atop \sum_{i=1}^{n+1} \lambda^i = 1}  \sum_{i=1}^{n+1} \lambda^i f(\bm{y}^{(i)},\bm{c}) \\
=& \min_{\bm{x}\in\X,\atop \bm{y}^{(1)},\ldots,\bm{y}^{(n+1)}\in\X(\bm{x})} \max_{\bm{c}\in\cU} \min_{i=1,\ldots,n+1} f(\bm{y}^{(i)},\bm{c}).
\end{align*}
\end{proof}

This approach can be used to derive a compact formulation to the uncertain single machine scheduling problem (RRS). To this end, we first consider the inner selection problem, given a first-stage solution $\bm{x}$ and set of recovery solutions $\bm{y}^1,\dots,\bm{y}^K$, and a scenario $\bm{p}$. This is given by 
\begin{align*}
	\min\ &\sum_{k\in\cK} \left( \sum_{i\in \N} \sum_{j\in \N} p_i(n+1-j) y^k_{ij} \right) \lambda_k \\
\text{s.t. } & \sum_{k\in\cK} \lambda_k = 1 \\
	     & \lambda_k \ge 0 & \forall k\in\cK,
\end{align*}
where $\cK = \{1,\ldots,K\}$. The problem of finding the worst-case scenario $\bm{p}\in\mathcal{U}$ for the choice of first-stage solution $\bm{x}$ and recovery solutions $\bm{y}^1,\dots, \bm{y}^K$ is therefore:
\begin{align*}
	\max\ & t \\
\text{s.t. } & t \le  \sum_{i\in \N} \sum_{j\in \N} p_i(n+1-j) y^k_{ij} & \forall k\in\cK \\
	     & \sum_{i\in \N} a_{mi} p_i \le b_m & \forall m\in \cM\\
& p_i \ge 0 & \forall i\in \N.
\end{align*}
Dualising this problem then gives the following formulation for (RRS):
\begin{align}
	\min\ & \sum_{m\in\cM} b_m q_m \label{general_formulation:first}\\
\text{s.t. } & \sum_{k\in\cK} \mu_k = 1 \\
	     & \sum_{m\in\cM} a_{m i} q_m \ge \sum_{k\in\cK} \left(\sum_{j\in \N}(n+1-j)y^k_{ij}\right) \mu_k & \forall i\in \N \\
	& d(\bm{x},\bm{y}^k) \le \Delta & \forall k\in\cK \\
& \bm{x}\in\X \\
& \bm{y}^k \in \X & \forall k\in\cK\\
& \mu_k \geq 0 & \forall k\in \cK\\
& q_m \geq 0 & \forall m \in \cM\label{general_formulation:last},
\end{align}
where $d(\bm{x},\bm{y}^k)$ is some measure of distance between the first-stage solution and the $k$-th recovery solution. Note that this model is not restricted to any particular choice of distance measure $d(\bm{x},\bm{y}^k)$. 

However, if we opt to calculate the distance between two schedules as the minimum number of disjoint pairwise swaps required to transform one schedule into the other, this can be modelled as follows. Let $z^k_{ii^{'}}$ denote the whether or not jobs $i$ and $i^{'}$ have swapped positions in recovery solution $\bm{y}^k$, relative to the first-stage schedule $\bm{x}$. In this case, we have that
$$y^k_{ij} = \sum_{i^{'}\in \N} z^k_{ii'} x_{i' j}.$$
Hence, $\bm{y}^k$ can be removed from the model, and replaced by $\bm{z}^k$ with the inclusion of the following constraints:
\begin{align*}
& \sum_{i^{'}\in \N} z^k_{ii^{'}} = 1 & \forall i\in \N, \, k\in\cK \\
& \sum_{i\in \N} z^k_{ii^{'}} = 1 & \forall i^{'}\in \N, \, k\in\cK  \\
& z^k_{ii^{'}} = z^k_{i^{'}i} & \forall i,i^{'}\in \N, \, k\in\cK  \\
& \sum_{i\in \N} z^k_{ii} \ge n - 2\Delta & \forall k\in\cK.
\end{align*}
To arrive at a mixed-integer linear program, the products $z^k_{ii'}\cdot x_{i' j}\cdot \mu_k$ need to be linearised using standard techniques. The full linearised formulation contains $O(n^3K)$ constraints and variables and is shown in Appendix~\ref{sec:model1}.

\section{Complexity of subproblems and compact formulations}
\label{sec:complexity}

In this section, we examine the incremental and adversarial problems of (RRS) in more detail and subsequently derive two additional compact formulations. 

\subsection{Matching-based formulation}\label{subsection:matching_formulation}

We first consider a matching-based formulation for the incremental problem. For the ease of presentation, we assume for now that $x_{ii} = 1$ for all $i\in \N$, i.e. the first-stage solution is a horizontal matching. Note a change in notation for this section where now the indices $i$ and $j$ are both used to refer to jobs, and $\ell$ denotes a position in the schedule. Supposing that the positions of jobs $i$ and $j$ are switched in the recovery schedule, the reduction in cost of making this switch is given by
$$p_i(n+1-i)+p_j(n+1-j)-p_i(n+1-j)-p_j(n+1-i)=(p_i-p_j)(j-i).$$

Letting $z_{ij}$ indicate whether or not jobs $i$ and $j$ swap positions in the schedule, the incremental problem can be formulated as:

\begin{align}
\min\ & \sum_{i\in \N}p_i(n+1-i)-\sum_{e=\{i,j\}\in \cE}(p_i-p_j)(j-i)z_e\label{matchingincr1}\\
\textnormal{s.t. }&\sum_{e\in\delta(i)}z_e \leq 1 &\forall i \in \N\\
&\sum_{e\in \cE}z_e \leq \Delta\\
&z_e\in\{0,1\}&\forall e\in \cE\label{matchingincr4},
\end{align}
where $\cE=\{\{i,j\}:i,j\in \N,\,i\neq j\}$ is the set of unique swaps, and $\delta(i)$ is the set of edges incident to vertex $i$. This is a cardinality-constrained matching problem on a complete graph with one node for each job $i\in \N$.


We examine this matching-based formulation in further detail. First, consider the maximum weight matching problem on a general graph $G=(\cV,\cE)$. This problem can be formulated as the following linear program:
\begin{align}
	\max & \sum_{e\in \cE} w_{e} x_{e} \label{matching1} \\
	\text{s.t. } & \sum_{e\in \delta(i)} x_{e} \leq 1 & \forall i \in \cV\label{matching2}\\
		     & \sum_{e\in \cE(\cW)}x_{e} \leq \frac{|\cW|-1}{2} & \forall \cW\subseteq \cV,\, |\cW| \textnormal{ odd}\label{matching3}\\
		     & x_{e} \geq 0 & \forall e \in \cE, \label{matching4}
\end{align}
where $\cE(\cW)$ is the set of edges in the subgraph induced on $\cW$. \cite{edmonds1965maximum} showed that constraints (\ref{matching3}), known as odd-cycle constraints or blossom constraints, are required to fully characterise the matching polytope.

In the following theorem, we show that for a matching problem with the same cost structure as \eqref{matchingincr1}, odd-cycle constraints are not required.

\begin{theorem}\label{theorem:matching}
	For any $\bm{a},\,\bm{b} \in \mathbb{R}^{|\cV|}_+$, the problem
	\begin{align}
		\max &\sum_{e=\{i,j\}\in \cE}(a_i-a_j)(b_i-b_j)x_e \label{obj:thm1}\\
		\text{s.t. } & \sum_{e\in\delta(i)} x_e \leq 1 & \forall i \in \cV\label{cons:thm1}\\
			     & x_e \geq 0 & \forall e \in \cE\label{cons:thm2}
	\end{align}
	has an optimal solution with $x_e\in \{0,1\}$ for all $e\in \cE$.
\end{theorem}

\begin{proof}
\citet[Theorem 30.2, page 522]{schrijver2003combinatorial} states that each vertex of the matching polytope described by (\ref{cons:thm1}) and (\ref{cons:thm2}) is half-integer, i.e. $x_e\in\{0,\frac{1}{2},1\}$ for all $e\in \cE$ in an optimal solution. Additionally, as observed by \cite{balinski1965integer}, the vertices of the matching polytope can be partitioned into a matching $\mathcal{P}$, where $x_e=1$ for each $e\in \mathcal{P}$, and a set of 1/2-fractional cycles of odd length, where $x_e=\frac{1}{2}$ for each $e$ in the odd cycles. Hence, we can restrict our attention only to 1/2-fractional odd cycles, and show that there is an optimal solution where such cycles do not exist.

Suppose we are given an optimal solution containing a 1/2-fractional odd cycle, consisting of edges $\cC=\{e_{i_1,i_2},\,e_{i_2,i_3},\,\dots,\,e_{i_{q-1},i_q},\,e_{i_q,i_1}\}$, with weights given by $w_{ij}=(a_i-a_j)(b_i-b_j)$. Without loss of generality, we assume an orientation in the cycle, where edges are directed as $(i_j,i_{j+1})$ for $j=1,\ldots,q$, where $i_{q+1}=i_1$.

Note that if $w_e\leq 0$ for some edge $e$, it can be removed from $\cE$, as such an edge will never be selected in an optimal matching. Hence, we may assume that $w_e> 0$ for all $e\in \cC$.
Since $w_{ij}=(a_i-a_j)(b_i-b_j)>0$ for all $e_{ij}\in \cC$, $(a_i-a_j)$ and $(b_i-b_j)$ must have the same sign. That is, either $a_i>a_j$ and $b_i>b_j$, in which case we refer to $e_{ij}$ as a \textit{decreasing edge}, or $a_i<a_j$ and $b_i<b_j$, in which case we refer to $e_{ij}$ as an \textit{increasing edge}.

We show that there is an optimal 1/2-fractional cycle that alternates between increasing and decreasing edges. Suppose that there are $p<q$ consecutive decreasing edges in $\cC$, $e_{j_1,j_2},e_{j_2,j_3},\dots,e_{j_{p-1},j_p}$, i.e. $a_{j_1}> a_{j_2}> \dots >a_{j_p}$ and $b_{j_1}> b_{j_2} > \dots > b_{j_p}$. In this case 
\begin{align*}
w_{j_1,j_p}&=(a_{j_1}-a_{j_p})(b_{j_1}-b_{j_p})\\
&=\Big((a_{j_1}-a_{j_2})+(a_{j_2}-a_{j_3})+\dots+(a_{{j_p-1}}-a_{j_p})\Big) \\
& \qquad \cdot\Big((b_{j_1}-b_{j_2})+(b_{j_2}-b_{j_3})+\dots+(b_{j_{p-1}}-b_{j_p})\Big)\\
&=w_{j_1,j_2} + w_{j_2,j_3} + \dots + w_{j_{p-1},j_p} + (a_{j_1}-a_{j_2})\Big((b_{j_2}-b_{j_3})+\dots+(b_{j_{p-1}}-b_{j_p})\Big)\\
&\hspace{4.6cm}+ (a_{j_2}-a_{j_3})\Big((b_{j_1}-b_{j_2})+\dots+(b_{j_{p-1}}-b_{j_p})\Big)\\
&\hspace{4.6cm}+\dots\\
&\hspace{4.6cm}+ (a_{j_{p-1}}-a_{j_p})\Big((b_{j_1}-b_{j_2})+\dots+(b_{j_{p-2}}-b_{j_{p-1}})\Big)\\
&> w_{j_1,j_2}+w_{j_2,j_3}+\dots+w_{j_{p-1},j_p},
\end{align*}
which means that replacing the $p$ consecutive decreasing edges in $\cC$ by the edge $e_{j_1,j_p}$ would lead to an even better objective value (see Figure~\ref{fig:proof} for an illustration). The same argument can be used to show that there also cannot be $p$ consecutive increasing edges in an optimal 1/2-fractional cycle.

We have therefore constructed an optimal 1/2-fractional cycle that strictly alternates between increasing and decreasing edges. Clearly, this is only possible if $q$ is even. Since a 1/2-fractional even cycle can be written as a convex combination of two feasible matchings, this proves that exists an optimal solution without any 1/2-fractional cycles.
\end{proof}

\begin{figure}[h]
	\centering
	\includegraphics[scale=0.75]{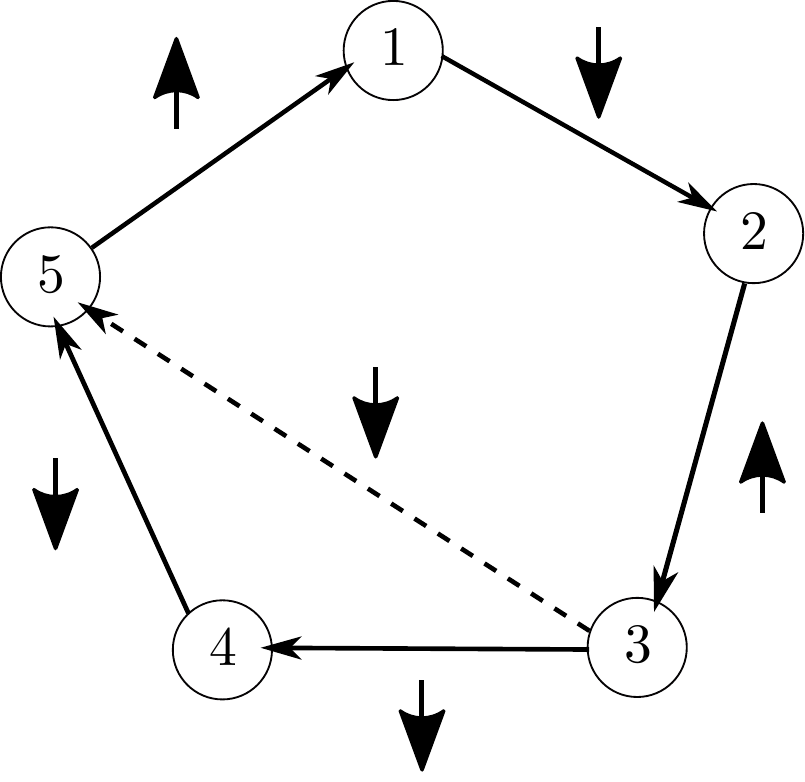}
	\caption{An example of a 1/2-fractional cycle involving $q=5$ nodes. Up and down arrows indicate increasing and decreasing edges respectively. It is optimal to replace the two consecutive decreasing edges $(3,4)$ and $(4,5)$ with the dashed edge $(3,5)$, i.e. $w_{35}>w_{34}+w_{45}$.}
	\label{fig:proof}
\end{figure}

The following result, presented in \cite{schrijver2003combinatorial} (Corollary~18.10a, page~331), states that the integrality of the vertices of the matching polytope is unaffected by the addition of a cardinality constraint.

\begin{theorem} Let $G=(\cV,\cE)$ be an undirected graph and let $k,l\in \mathbb{Z}_+$ with $k\leq l$. Then the convex hull of the incidence vectors of matchings $\mathcal{P}$ satisfying $k\leq |\mathcal{P}| \leq l$ is equal to the set of those vectors $\bm{x}$ in the matching polytope of $G$ satisfying $k\leq \bm{1}^\top \bm{x}\leq l$.
\end{theorem}

This result, in combination with Theorem \ref{theorem:matching}, provides us with the following corollary:
\begin{cor}\label{cor:ccmatching}
	For any $\bm{a},\,\bm{b} \in \mathbb{R}^{|\cV|}_+$, the problem
	\begin{align}
		\max &\sum_{e=\{i,j\}\in \cE}(a_i-a_j)(b_i-b_j)x_e\label{matchinc1}\\
		\text{s.t. } & \sum_{e\in\delta(i)} x_e \leq 1 & \forall i \in \cV\label{matchinc2}\\
			     & \sum_{e\in \cE} x_e\leq \Delta\label{matchinc3}\\
			     & x_e \geq 0 & \forall e \in \cE\label{matchinc4}
	\end{align}
	has an optimal solution with $x_e\in \{0,1\}$ for all $e\in \cE$.
\end{cor}

Hence, given a first-stage solution $\bm{x}$ and scenario $\bm{p}$, we can formulate the incremental problem as a linear program with polynomially many constraints. We use this result to derive a compact formulation for the full uncertain single machine scheduling problem (RRS).

We begin by formulating the incremental problem $\text{Inc}(\bm{x},\bm{p})$ according to Corollary~\ref{cor:ccmatching}. Note that we now consider a general first-stage assignment that is not necessarily horizontal, and therefore introduce terms $\sum_{\ell\in \N}\ell\cdot x_{i\ell}$ to track the position in which job $i$ is scheduled in the first-stage schedule. We fix an arbitrary orientation of edges, using $\cE=\{ (i,j)\in \N \times \N : i< j\}$ in the following.
\begin{align*}
	\min_{\bm{z}}\ &\sum_{i\in \N}p_i\Bigg(n+1-\sum_{\ell\in \N}\ell \cdot x_{i\ell}\Bigg)-\sum_{(i,j)\in \cE}(p_i-p_j)\Bigg(\sum_{\ell\in \N}\ell \cdot x_{j\ell}-\sum_{\ell\in \N}\ell \cdot x_{i\ell}\Bigg)z_{ij}\\
	\textnormal{s.t. }&\sum_{(i,j)\in \cE}z_{ij}+\sum_{(j,i)\in \cE}z_{ji} \leq 1 &\hspace{-2cm}\forall i \in \N\\
			  &\sum_{(i,j)\in \cE}z_{ij}\leq \Delta\\
			  &z_{ij}\geq 0&\hspace{-2cm}\forall (i,j)\in \cE.
\end{align*}
Taking the dual of this, we get the following formulation for the adversarial problem $\text{Adv}(\bm{x})$:
\begin{align*}
	\max_{\bm{p},\,\bm{\alpha},\,\gamma}\ &\sum_{i\in \N}p_i\Bigg(n+1-\sum_{\ell\in \N}\ell\cdot x_{i\ell}\Bigg) -\sum_{i\in \N}\alpha_i-\gamma\Delta\\
	\textnormal{s.t. }&\alpha_i + \alpha_j + \gamma \geq (p_i-p_j)\Bigg(\sum_{\ell\in \N}\ell\cdot x_{j\ell}-\sum_{\ell\in \N}\ell\cdot x_{i\ell}\Bigg)& \forall (i,j) \in \cE \\
			  &\sum_{i\in \N}a_{mi}p_i \leq b_m & \forall m\in\cM\\
			  &p_i \geq 0 &\forall i \in \N\\
			  &\alpha_i \geq 0 &\forall i \in \N\\
			  &\gamma \geq 0.
\end{align*}
Since this is a linear program, we immediately obtain following result:
\begin{cor}
The adversarial problem can be solved in polynomial time. 
\end{cor}

Finally, dualising the above adversarial formulation, we get the following compact formulation for problem (RRS):
\begin{align}
	\min_{\bm{x},\,\bm{z},\, \bm{q}}\ &\sum_{m\in\cM} b_mq_m \label{matching_model:first}\\
	\textnormal{s.t. }&\sum_{m\in\cM} a_{mi}q_m+\sum_{(i,j)\in \cE}\Bigg(\sum_{\ell\in \N}\ell\cdot x_{j\ell}-\sum_{\ell\in \N}\ell\cdot x_{i\ell}\Bigg)z_{ij}\nonumber \\
			  &\hspace{0.55cm}-\sum_{(j,i)\in \cE}\Bigg(\sum_{\ell\in \N}\ell \cdot x_{i\ell}-\sum_{\ell\in \N}\ell\cdot x_{j\ell}\Bigg)z_{ji} \geq (n+1-\sum_{\ell \in \N}\ell\cdot x_{i\ell})&\forall i\in \N \label{matching_model:constr1}\\
			    &\sum_{(i,j)\in \cE}z_{ij}+\sum_{(j,i)\in \cE}z_{ji}\leq 1 & \forall i\in \N \label{matching_model:constr2}\\
			    &\sum_{(i,j)\in \cE}z_{ij}\leq \Delta \label{matching_model:constr3}\\
	&x\in \mathcal{X}\\
	&q_m\ge 0 &\hspace{-3cm}\forall m\in \cM\\
	&z_{ij}\geq 0&\hspace{-2cm}\forall (i,j)\in \cE \label{matching_model:last}.
\end{align}
Upon linearising the quadratic $x_{i\ell}\cdot z_{ij}$ and $x_{j\ell}\cdot z_{ij}$ terms, this model becomes a mixed-integer linear program. The fully linearised model contains $O(n^3)$ constraints and variables and is presented in full in Appendix~\ref{sec:model2}.
 
\subsection{Assignment-based formulation} \label{subsection:assignment_formulation}

We now consider an alternative formulation for the incremental problem. Again, for the purposes of examining the incremental problem, we initially consider the first-stage schedule to be a horizontal assignment, i.e. $x_{ii}=1$ for all $i\in \N$. By letting variables $y_{ij}$ represent a second-stage assignment (we now return to the convention where the index $i$ is used to denote a job and the index $j$ is used to denote a position in the schedule), we can formulate the incremental problem as follows:
\begin{align}
\min \ &\sum_{i\in \N} \sum_{j\in \N} p_i (n+1-j) y_{ij} \label{inc1}\\
\text{s.t. } & \sum_{i\in \N} y_{ij} = 1 & \forall j\in \N \label{inc2} \\
& \sum_{j\in \N} y_{ij} = 1 & \forall i\in \N \label{inc3}\\
& y_{ij} = y_{ji} & \forall i,j\in \N \label{inc4}\\
& \sum_{i\in \N} y_{ii} \ge n-2\Delta \label{inc5}\\
& y_{ij} \in \{0,1\} & \forall i,j\in \N \label{inc6}.
\end{align}
Constraints (\ref{inc4}) and (\ref{inc5}) ensure that the second-stage assignment is a feasible recovery to the first-stage solution, that is, the second-stage assignment is constructed by swapping the first-stage positions of up to $\Delta$ disjoint pairs of jobs. Note that this is a level-constrained symmetric perfect matching problem, which can be solved in polynomial time \cite[Theorem~2.28]{thomas2016matching}.

We show that problem (\ref{inc1})-(\ref{inc6}) can be solved as a linear program as a result of its non-general cost structure. As the proof is technical and based on a reduction to the corresponding maximum weight matching problem, it is omitted here and can be found in Appendix~\ref{section:proofs}.

\begin{theorem}\label{theoremlp}
For any $\bm{a},\,\bm{b}\in\mathbb{R}^n_+$, the problem
\begin{align}
	\min\ & \sum_{i\in \N} \sum_{j\in \N} a_i b_j y_{ij} \label{assign1} \\
	\text{s.t. } & \eqref{inc2}-\eqref{inc5} \label{assign2}\\
		     & y_{ij} \ge 0 \label{assign3}
\end{align}
has an optimal solution with $y_{ij}\in\{0,1\}$ for all $i,j\in \N$.
\end{theorem}
We now use this result to find an assignment-based formulation for (RRS). We first write the incremental problem in the form given by \eqref{assign1}-\eqref{assign3}. Since we are now considering the case where $\bm{x}$ is not necessarily a horizontal matching, we rearrange the indices accordingly. 
\begin{align*}
	\min \ &\sum_{i\in \N} \sum_{j\in \N} p_i (n+1-\sum_{\ell\in \N}\ell \cdot x_{j\ell}) y_{ij}\\
\text{s.t. } & \sum_{i\in \N} y_{ij} = 1 & \forall j\in \N  \\
& \sum_{j\in \N} y_{ij} = 1 & \forall i\in \N \\
& y_{ij} = y_{ji} & \forall i,j\in \N \\
& \sum_{i\in \N} y_{ii} \ge n-2\Delta \\
& y_{ij} \geq 0 & \forall i,j\in \N .
\end{align*}
Taking the dual of this, the adversarial problem can be formulated in the following way:
\begin{align*}
	\max_{\bm{\alpha},\,\bm{\beta},\,\bm{\gamma}, \tau,\,\bm{p}}\ & \sum_{i\in \N} (\alpha_i + \beta_i) + (n-2\Delta)\tau \\
\text{s.t. } & \alpha_j + \beta_i + \gamma_{ij} \le (n+1 - \sum_{\ell\in \N}\ell\cdot x_{j\ell} ) p_i & \forall i,j\in \N : i<j \\
& \alpha_j + \beta_i - \gamma_{ji}\le (n+1 - \sum_{\ell\in \N}\ell\cdot x_{j\ell})p_i & \forall i,j\in \N : i>j \\
& \alpha_i + \beta_i + \tau \le (n+1 - \sum_{\ell\in \N}\ell\cdot x_{j\ell})p_i & \forall i\in \N \\
& \sum_{i\in \N} a_{mi}p_i\leq b_m & \forall m\in\cM \\
& p_i \geq 0 & \forall i\in \N\\
& \tau \ge 0.
\end{align*}
Finally, we dualise this adversarial formulation to derive the following formulation for the recoverable problem:
\begin{align}
	\min_{\bm{x},\,\bm{y},\,\bm{q}}\ & \sum_{m\in\cM} b_mq_m  \label{assignment_model:first}\\
	\text{s.t. } & \sum_{i\in \N} y_{ij} = 1 & \forall j\in \N \label{assignment_model:assignment_constr1}\\
& \sum_{j\in \N} y_{ij} = 1 & \forall i\in \N \label{assignment_model:assignment_constr2}\\
& \sum_{i\in \N} y_{ii} \ge n-2\Delta \label{assignment_model:constr1}\\
& y_{ij} = y_{ji} & \forall i,j\in \N \label{assignment_model:constr2}\\
& \sum_{m\in\cM} a_{mi}q_m \ge \sum_{j\in \N} (n+1-\sum_{\ell\in \N}\ell \cdot x_{j\ell})y_{ij} & \forall i\in \N\label{assignment_model:constr3}\\
& \bm{x} \in \X \\
& q_m \geq 0 & \forall m\in\cM \\
& y_{ij} \ge 0 & \forall i,j\in \N \label{assignment_model:last}
\end{align}
As before, products $x_{j\ell}\cdot y_{ij}$ can be linearised using standard techniques. The resulting mixed-integer linear program contains $O(n^3)$ constraints and variables and can be found in Appendix~\ref{sec:model3}.

\section{Comparison of formulations}

This section presents a brief investigation into the linear relaxations of the three formulations derived above in order to compare their relative theoretical strengths. We begin by showing that no comparisons can be made between the general formulation and the other two formulations.

In preparation of the proof of this result we introduce budgeted uncertainty as a special case of polyhedral uncertainty. A budgeted uncertainty set can be defined as 
$$\mathcal{U}_B=\left\{\bm{p}\in \mathbb{R}^n_+:
\sum_{i\in \N}\frac{p_i-\hat{p}_i}{\bar{p}_i-\hat{p_i}}\leq \Gamma,\,
p_i\in[\hat{p}_i, \hat{p}_i+\bar{p}_i],\,i\in \N\right\},$$
where $\hat{p}_i$ is the nominal processing time of job $i$ and $\bar{p}_i$ is the worst-case delay to the processing time of job $i$. Introduced by \cite{bertsimas2004price}, its motivation is to exclude unrealistically pessimistic worst-case scenarios from the uncertainty set and thereby avoid overly conservative and highly-expensive solutions. This is achieved by assuming that at most $\Gamma$ jobs can simultaneously reach their maximum delays. Note that when $\Gamma=0$, each job assumes its nominal processing time and the $\mathcal{U}_B$ reduces to a single scenario. Additionally, observe that as $\Gamma \rightarrow n$, this budgeted uncertainty set becomes an interval. When $\Gamma = n$ the worst-case scenario is known a priori to be when all jobs achieve their worst-case processing times $\hat{p}_i+\bar{p}_i$. In this case the problem can be solved by simply ordering the jobs accoring to their worst-case processing times, and no recourse action will be required. The proof of the following proposition makes use of an instance involving a bugdeted uncertainty set.

\begin{theorem}\label{thm:incomparable_formulations}
	The general formulation (\ref{generalfull:first})-(\ref{generalfull:last}) is incomparable with both the matching-based formulation (\ref{matchingfull:first})-(\ref{matchingfull:last}) and the assignment-based formulation (\ref{assignmentfull:first})-(\ref{assignmentfull:last}).
\end{theorem}

\begin{proof}
First consider a problem with two jobs with processing times that lie in the uncertainty set $\mathcal{U}=\{(p_1,p_2):p_1 \leq 3,\, p_2\leq 3,\,p_1+2p_2\leq 7\}$. Suppose also that $\Delta=1$, i.e. one swap can be made to amend the first-stage schedule. In this case, the linear relaxation of the matching-based formulation has an objective value of 7, whilst the linear relaxation of the assignment-based formulation has an objective value of 5.  

Now consider an instance involving jobs with $\hat{\bm{p}}=(10,8,9,4,1,5,7,1)$ and $\bar{\bm{p}}=(9,7,5,4,1,3,6,1)$ lying in the budgeted uncertainty set $\mathcal{U}_B$, and set $\Gamma=1$ and $\Delta=1$. The linear relaxation of the matching-based formulation for this instance has an optimal objective value of -9.2 (to 1 decimal place), whilst the linear relaxation of the assignment-based formulation has an objective value of -285.4 (to 1 decimal place).

	For both of these instances, the linear relaxation of the general formulation attains an objective value of 0. (In fact, for any polyhedral uncertainty set in which $a_{mi} \ge 0$ and $b_m \ge 0$ for all $i\in\N$, the linear relaxation of the general formulation will be 0, since it is free to set $h_{ii'j}^k=0$ for all $i,\,i',\,j\in \N,\,k\in\cK$ and therefore $q_m = 0$ for all $m\in\cM$.)
	
	These examples show that the matching and assignment-based formulations are tighter than the general formulation for some instances, but less tight for other instances. Hence the general formulation is incomparable with the matching and assignment-based formulations.

\end{proof} 

It is the case however that the objective value of the linear relaxation of the non-linear matching-based formulation is always greater than or equal to the objective value for the linear relaxation of the non-linear assignment-based formulation. That is, that non-linear matching formulation dominates the non-linear assignment formulation.

\begin{theorem}\label{prop:dominance}
	The non-linear matching-based formulation (\ref{matching_model:first})-(\ref{matching_model:last}) dominates the non-linear assignment-based formulation (\ref{assignment_model:first})-(\ref{assignment_model:last}).
\end{theorem}

The proof of this statement involves the construction of a transformation $\phi$ to show that any feasible solution to the matching formulation can be transformed into a feasible solution to the assignment problem. The proof can be found in Appendix~\ref{section:proofs}. It does however remain open as to whether this result can be extended to the linearised versions of these formulations given by (\ref{matchingfull:first})-(\ref{matchingfull:last}) and (\ref{assignmentfull:first})-(\ref{assignmentfull:last}), respectively.

\section{Computational experiments}
\label{section:computational_experiments}

This section presents and compares results from solving the three compact models introduced in this paper, as well as three additional heuristic solution methods. As a particular example of a general polyhedral uncertainty, here we consider budgeted uncertainty as outlined in the previous section. Before introducing the heuristics we propose for solving this problem and examining their performance, we comment on the test instances and computational hardware used for these experiments.

Instances have been generated by randomly sampling both $\hat{p}_i$ and $\bar{p}_i$ from the set $\{1,2,\dots,100\}$. 20 instances of sizes $n\in\{10,15,20\}$ have been generated, resulting in a total of 60 deterministic test instances. For each deterministic instance, three uncertain instances have been generated by setting $\Gamma\in\{3,5,7\}$, resulting in a total of 180 uncertain instances. These instances, as well as the complete results data, can be found at \url{https://github.com/boldm1/RR-single-machine-scheduling}.

All methods have been run on 4 cores of a 2.30GHz Intel Xeon CPU, limited to 16GB RAM. The exact models have been solved using Gurobi 9.0.1, with a time limit of 10 minutes.

\subsection{Heuristics}

The three heuristic methods we consider are as follows:

\begin{enumerate}
	\item \textbf{Sorting}. Obtain a schedule by ordering the jobs $i\in \N$ according to non-decreasing $\hat{p}_i+\bar{p}_i$, i.e. a schedule that performs best in the worst-case scenario when $\Gamma=n$, and evaluate by solving $\text{Adv}(\bm{x})$.
	\item \textbf{Max-min}. Solve the max-min problem 
	\[ \max_{\bm{p}\in\cU_B} \min_{\bm{x}\in\X} \sum_{i\in \N} \sum_{j\in \N} p_i(n+1-j)x_{ij} \]
	to obtain a worst-case scenario $\bm{p}\in \mathcal{U}_B$. Find a schedule $\bm{x}$ that performs best in this worst-case scenario and evaluate by solving $\text{Adv}(\bm{x})$. 
	\item \textbf{Min-max}. Solve the problem without recourse, i.e. with $\Delta=0$.
\end{enumerate}

Each heuristic method has been used to find a feasible solution to all 180 uncertain test instances. Figure \ref{fig:heuristic_best_gaps} shows the cumulative percentage of instances solved by each of the heuristics to within a given gap to the best solution found by any method, including the exact models, which have been solved with $\Delta=2$. It is clear from this plot that min-max is the strongest of the three proposed heuristics, solving all 180 instances to within 3.2\% of the best solution. This gap increases to 8.8\% for sorting, whilst max-min solves all but one instance to within 15\%. The average of these gaps across all instances for min-max, sorting and max-min are 0.9\%, 3.1\% and 4.1\% respectively. 

Given its strong performance, we propose using min-max to provide a warm-start solution to the exact models. The benefits of this are assessed in the next section.

\begin{figure}[htb]
	\centering
	\includegraphics[scale=0.7]{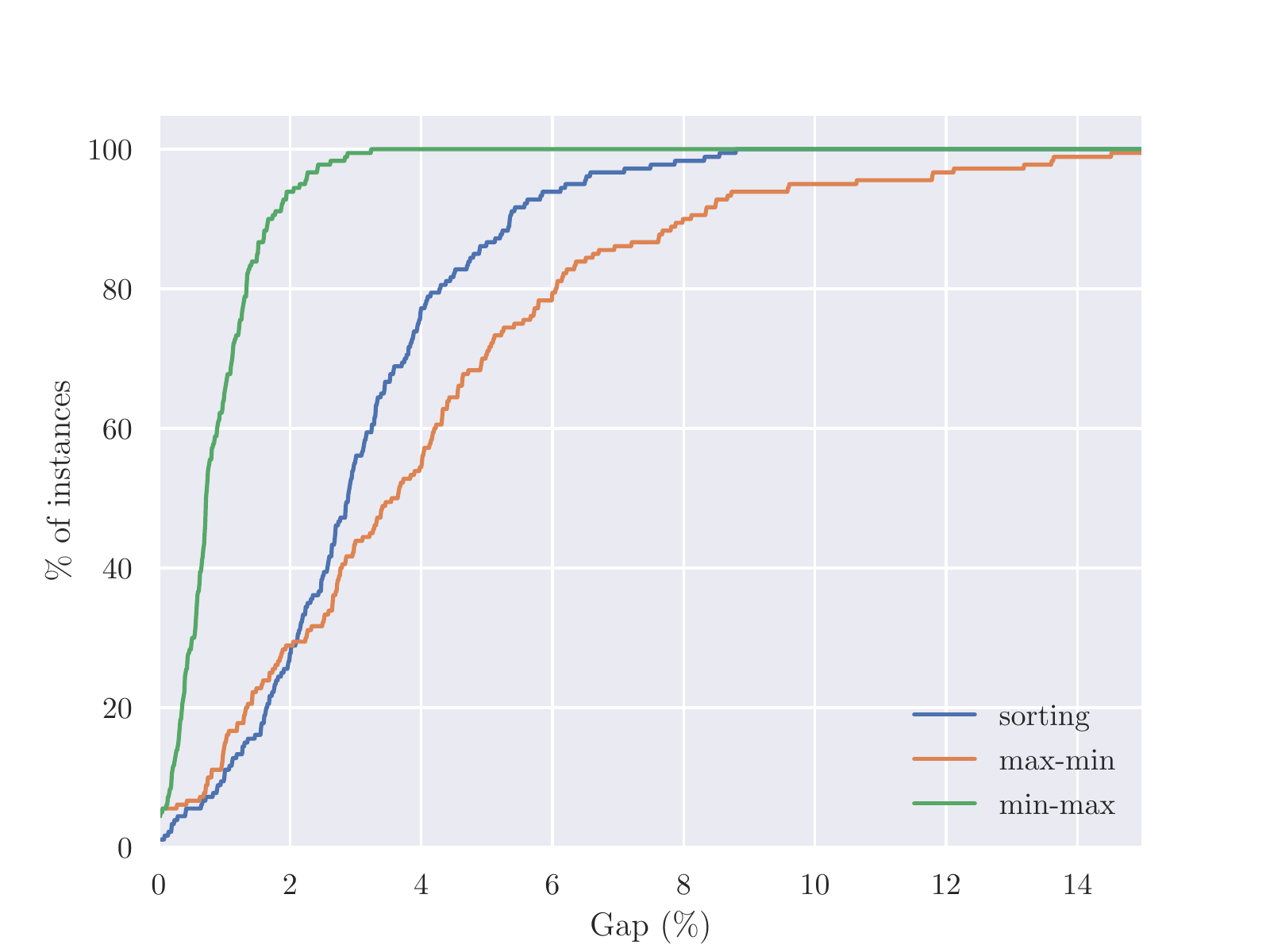}
	\caption{Cumulative percentage of instances solved to within a given gap of the best known solution.}
	\label{fig:heuristic_best_gaps}
\end{figure}

\subsection{Exact models}

We now examine the results of solving the three exact models proposed in this paper and their warm-start variants. The 180 uncertain instances have been solved by each model and its warm-start variant for $\Delta\in\{0,1,2,3\}$. Note that the general model has been implemented with $K=2$. This has been chosen to make the general model as computationally efficient to solve as possible, whilst actually still providing an advantage over the min-max model, i.e. for $K=1$ the general model corresponds to the min-max model.

Tables \ref{table:Gamma} and \ref{table:Delta} compare the performance of these exact models for different values of $\Gamma$ and $\Delta$ respectively. For each set of 20 instances with the same combination of instance parameters, Tables \ref{table:Gamma} and \ref{table:Delta} report the following:
\begin{itemize}
	\item \textit{time} - Average CPU time (secs) required to solve the instances that were solved to optimality within the time limit.
	\item \textit{LBgap} - Average gap (\%) between the best objective bound and the best known feasible solution found by any method, over the instances not solved to optimality within the time limit.
	\item \textit{UBgap} - Average gap (\%) between the best feasible solution found within the time limit and the best known feasible solution found by any method, over the instances not solved to optimality within the time limit.
	\item $\#$\textit{solv} - Number of instances solved to optimality within the time limit.
\end{itemize}

From Tables \ref{table:Gamma} and \ref{table:Delta}, it is clear that the general model is by far the weakest of the three proposed models. Other than for $\Delta=0$, no instances are solved to optimality. The general model is able to find near-optimal feasible solutions, but fails to begin closing the optimality gap in most instances. The matching-based model improves considerably on the general model, whilst the assignment model is the strongest performing of the three exact models, solving the most number of instances to optimality and having the smallest gaps over those instances that cannot be solved to optimality. The addition of a warm-start solution is clearly beneficial only for the assignment-based model, where the addition solves more instances to optimality in less time. 

From Table \ref{table:Gamma} it can be seen that instances tend to become harder to solve as $\Gamma$ increases from 3 to 7. From Table \ref{table:Delta} we observe that, unsurprisingly, instances are easiest to solve to solve when $\Delta=0$ (this corresponds to solving the min-max model). Interestingly however, when $n=15$ and $n=20$, instances are most difficult when $\Delta=1$, and become easier to solve as the number of recovery swaps allowed, $\Delta$, increases, i.e. the second stage-solution becomes less constrained by the first-stage solution.

\begin{table}
\centering
\begin{tabular}{lllrrrrcrrrr}
\hline \hline
	     &          &          & \multicolumn{4}{c}{General}                      & & \multicolumn{4}{c}{General + warm-start}                     \\
   \cline{4-7} \cline{9-12}
	$n$  & $\Gamma$ & $\Delta$ & \textit{time}  & \textit{LBgap} & \textit{UBgap} & $\#$\textit{solv} & & \textit{time}  & \textit{LBgap} & \textit{UBgap} & $\#$\textit{solv}  \\
\hline\hline
10 & 3        & 2        & -          & 100.0       & 0.2   & 0            &        & -      & 100.0       & 0.2   & 0        \\
10 & 5        & 2        & -          & 100.0       & 0.1   & 0      	 &        & -      & 100.0       & 0.1   & 0        \\
10 & 7        & 2        & -          & 100.0       & 0.1   & 0      	 &        & -      & 100.0       & 0.0   & 0        \\
15 & 3        & 2        & -          & 100.0       & 0.4   & 0      	 &        & -      & 100.0       & 0.3   & 0        \\
15 & 5        & 2        & -          & 100.0       & 0.6   & 0      	 &        & -      & 100.0       & 0.3   & 0        \\
15 & 7        & 2        & -          & 100.0       & 0.4   & 0      	 &        & -      & 100.0       & 0.3   & 0        \\
20 & 3        & 2        & -          & 100.0       & 0.8   & 0      	 &        & -      & 100.0       & 0.3   & 0        \\
20 & 5        & 2        & -          & 100.0       & 0.9   & 0      	 &        & -      & 100.0       & 0.4   & 0        \\
20 & 7        & 2        & -          & 100.0       & 1.2   & 0      	 &        & -      & 100.0       & 0.5   & 0        \\
   \cline{4-7} \cline{9-12}
   &          &          &          &         &         & 0      	 &        &        &         &         & 0        \\
   &          &          & \multicolumn{4}{c}{Matching}                      	 &        & \multicolumn{4}{c}{Matching + warm-start}                     \\
   \cline{4-7} \cline{9-12}
	$n$  & $\Gamma$ & $\Delta$ & \textit{time} & \textit{LBgap} & \textit{UBgap} & $\#$\textit{solv} & & \textit{time} & \textit{LBgap} & \textit{UBgap} & $\#$\textit{solv}  \\
\hline\hline
10 & 3        & 2        & 2.6     & -       & -       & 20     	 &        & 2.6    & -       & -       & 20       \\
10 & 5        & 2        & 4.1     & -       & -       & 20     	 &        & 4.1    & -       & -       & 20       \\
10 & 7        & 2        & 2.1     & -       & -       & 20     	 &        & 3.4    & -       & -       & 20       \\
15 & 3        & 2        & 110.9   & 0.0     & 0.0     & 19     	 &        & 110.0  & -       & -       & 20       \\
15 & 5        & 2        & 132.4   & 0.2     & 0.0     & 17     	 &        & 110.0  & 0.1     & 0.0     & 18       \\
15 & 7        & 2        & 111.7   & 0.2     & 0.0     & 16     	 &        & 97.6   & 0.2     & 0.0     & 16       \\
20 & 3        & 2        & 540.7   & 0.9     & 0.0     & 3      	 &        & 463.7  & 0.9     & 0.0     & 2        \\
20 & 5        & 2        & 598.6   & 0.5     & 0.0     & 1      	 &        & -      & 0.6     & 0.0     & 0        \\
20 & 7        & 2        & -       & 0.5     & 0.0     & 0      	 &        & -      & 0.4     & 0.0     & 0        \\
   \cline{4-7} \cline{9-12}
   &          &          &         &         &         & 116    	 &        &        &         &         & 116      \\ 
   &          &          & \multicolumn{4}{c}{Assignment}                      	 &        & \multicolumn{4}{c}{Assignment + warm-start}                     \\
   \cline{4-7} \cline{9-12}
	$n$  & $\Gamma$ & $\Delta$ & \textit{time}  &  \textit{LBgap} & \textit{UBgap} & $\#$\textit{solv} & & \textit{time}  & \textit{LBgap} & \textit{UBgap} & $\#$\textit{solv}  \\
\hline\hline
10 & 3        & 2        & 8      & -       & -       & 20     	 &        & 8.2    & -       & -       & 20       \\
10 & 5        & 2        & 6.8    & -       & -       & 20     	 &        & 5.1    & -       & -       & 20       \\
10 & 7        & 2        & 3.4    & -       & -       & 20     	 &        & 3.0    & -       & -       & 20       \\
15 & 3        & 2        & 50.9   & -       & -       & 20     	 &        & 55.4   & -       & -       & 20       \\
15 & 5        & 2        & 62.6   & -       & -       & 20     	 &        & 59.3   & -       & -       & 20       \\
15 & 7        & 2        & 59.1   & -       & -       & 20     	 &        & 64.3   & -       & -       & 20       \\
20 & 3        & 2        & 344.0  & -       & -       & 20     	 &        & 223.7  & 0.9     & 0       & 19       \\
20 & 5        & 2        & 377.7  & -       & -       & 20     	 &        & 276.0  & 0.0     & 0       & 19       \\
20 & 7        & 2        & 445.5  & 0.0     & 0.0     & 19     	 &        & 321.3  & -       & -       & 20       \\
   \cline{4-7} \cline{9-12}
   &          &          &        &         &         & 179    	 &        &        &         &         & 178      \\ 
\hline \hline
\end{tabular}
\caption{Comparison of the three exact models proposed in this paper and their warm-start variants, for different values of $\Gamma$.}
\label{table:Gamma}
\end{table}

\begin{table}
\centering
\vspace*{-7ex}
\begin{tabular}{lllrrrrcrrrr}
	\hline\hline
   &       &       & \multicolumn{4}{c}{General}                    &  & \multicolumn{4}{c}{General + warm-start}                   \\ 
   \cline{4-7} \cline{9-12}
	$n$  & $\Gamma$ & $\Delta$ & \textit{time}  &  \textit{LBgap} & \textit{UBgap} & $\#$\textit{solv} & & \textit{time}  & \textit{LBgap} & \textit{UBgap} & $\#$\textit{solv}  \\
	\hline\hline
10 & 7     & 0     & 0.3    & -     & -      & 20      &  & 0.3     & -     & -     & 20         \\
10 & 7     & 1     & -      & 100.0   & 0.0    & 0       &  & -       & 100.0   & 0.0   & 0          \\
10 & 7     & 2     & -      & 100.0   & 0.1    & 0       &  & -       & 100.0   & 0.0   & 0          \\
10 & 7     & 3     & -      & 100.0   & 0.0    & 0       &  & -       & 100.0   & 0.0   & 0          \\
15 & 7     & 0     & 3.3    & -     & -      & 20      &  & 3.0     & -     & -     & 20         \\
15 & 7     & 1     & -      & 100.0   & 0.4    & 0       &  & -       & 100.0   & 0.3   & 0          \\
15 & 7     & 2     & -      & 100.0   & 0.4    & 0       &  & -       & 100.0   & 0.3   & 0          \\
15 & 7     & 3     & -      & 100.0   & 0.5    & 0       &  & -       & 100.0   & 0.3   & 0          \\
20 & 7     & 0     & 69.5   & 0.8   & 0.0    & 18      &  & 81.1    & 0.9   & 0.0   & 18         \\
20 & 7     & 1     & -      & 100.0   & 1.1    & 0       &  & -       & 100.0   & 0.7   & 0          \\
20 & 7     & 2     & -      & 100.0   & 1.2    & 0       &  & -       & 100.0   & 0.5   & 0          \\
20 & 7     & 3     & -      & 100.0   & 1.6    & 0       &  & -       & 100.0   & 0.5   & 0          \\
   \cline{4-7} \cline{9-12}
   &       &       &        &       &        & 58      &  &         &       &       & 58         \\
   &       &       & \multicolumn{4}{c}{Matching}                    &  & \multicolumn{4}{c}{Matching + warm-start}                   \\ 
   \cline{4-7} \cline{9-12}
	$n$  & $\Gamma$ & $\Delta$ & \textit{time}  &  \textit{LBgap} & \textit{UBgap} & $\#$\textit{solv} & & \textit{time}  & \textit{LBgap} & \textit{UBgap} & $\#$\textit{solv}  \\
	\hline\hline
10 & 7     & 0     & 0       & -     & -      & 20      &  & 0.0    & -     & -     & 20         \\
10 & 7     & 1     & 2.1     & -     & -      & 20      &  & 2.2    & -     & -     & 20         \\
10 & 7     & 2     & 2.1     & -     & -      & 20      &  & 3.4    & -     & -     & 20         \\
10 & 7     & 3     & 5.3     & -     & -      & 20      &  & 5.9    & -     & -     & 20         \\
15 & 7     & 0     & 0.2     & -     & -      & 20      &  & 0.2    & -     & -     & 20         \\
15 & 7     & 1     & 207     & 0.2   & 0.0    & 11      &  & 162.7  & 0.2   & 0.0   & 12         \\
15 & 7     & 2     & 111.7   & 0.2   & 0.0    & 16      &  & 97.6   & 0.2   & 0.0   & 16         \\
15 & 7     & 3     & 66.1    & 0.2   & 0.0    & 17      &  & 72.7   & 0.2   & 0.0   & 16         \\
20 & 7     & 0     & 1.6     & -     & -      & 20      &  & 1.6    & -     & -     & 20         \\
20 & 7     & 1     & -       & 0.9   & 0.0    & 0       &  & -      & 0.9   & 0.0   & 0          \\
20 & 7     & 2     & -       & 0.5   & 0.0    & 0       &  & -      & 0.5   & 0.0   & 0          \\
20 & 7     & 3     & 395.3   & 0.3   & 0.0    & 8       &  & 490.9  & 0.3   & 0.0   & 8          \\
\cline{4-7} \cline{9-12}
   &       &       &         &       &        & 172     &  &        &       &       & 172        \\ 
   &       &       & \multicolumn{4}{c}{Assignment}                    &  & \multicolumn{4}{c}{Assignment + warm-start}                   \\ 
   \cline{4-7} \cline{9-12}
	$n$  & $\Gamma$ & $\Delta$ & \textit{time}  &  \textit{LBgap} & \textit{UBgap} & $\#$\textit{solv} & & \textit{time}  & \textit{LBgap} & \textit{UBgap} & $\#$\textit{solv}  \\
	\hline\hline
10 & 7     & 0     & 0.0     & -     & -      & 20      &  & 0.0     & -     & -     & 20         \\
10 & 7     & 1     & 2.8     & -     & -      & 20      &  & 2.0     & -     & -     & 20         \\
10 & 7     & 2     & 3.4     & -     & -      & 20      &  & 3.0     & -     & -     & 20         \\
10 & 7     & 3     & 5.7     & -     & -      & 20      &  & 4.1     & -     & -     & 20         \\
15 & 7     & 0     & 0.2     & -     & -      & 20      &  & 0.1     & -     & -     & 20         \\
15 & 7     & 1     & 71.1    & 0.0   & 0.0    & 19      &  & 86.8    & -     & -     & 20         \\
15 & 7     & 2     & 59.1    & -     & -      & 20      &  & 64.3    & -     & -     & 20         \\
15 & 7     & 3     & 42.4    & -     & -      & 20      &  & 50.3    & -     & -     & 20         \\
20 & 7     & 0     & 1.5     & -     & -      & 20      &  & 1.3     & -     & -     & 20         \\
20 & 7     & 1     & 433.5   & 0.4   & 0.0    & 2       &  & 523.5   & 0.3   & 0.0   & 6          \\
20 & 7     & 2     & 445.5   & 0.0   & 0.0    & 19      &  & 321.3   & -     & -     & 20         \\
20 & 7     & 3     & 435.2   & 0.3   & 0.0    & 19      &  & 312.0   & -     & -     & 20         \\ 
\cline{4-7} \cline{9-12}
   &       &       &         &       &        & 219     &  &         &       &       & 226        \\
   \hline \hline
\end{tabular}
\caption{Comparison of the three exact models proposed in this paper and their warm-start variants, for different values of $\Delta$.}
\label{table:Delta}
\end{table}

Figure \ref{fig:models_ppcombined2} shows performance profiles \citep{dolan2002benchmarking} of the relative solution times of the matching and assignment-based models and their warm-start variants, for different instance sizes. The general model and its warm-start variant is excluded from these plots given its poor performance. A performance profile is a graphical comparison of the \textit{performance ratios}. The performance ratio of model $m\in \mathcal{M}$ for instance $i\in \mathcal{I}$ is defined as 
$$p_{im}=\frac{t_{im}}{\min_{m\in \mathcal{M}}t_{im}},$$
where $t_{im}$ is the time required to solve instance $i$ using model $m$. If model $m$ fails to find an optimal solution to instance $i$ within the given time-limit, then $p_{im}=P$, for some $P>\max_{i,m}r_{im}$. The performance profile of model $m\in \mathcal{M}$ is then defined to be the function
$$\rho_{m}(\tau)=\frac{|\{p_{im}\leq \tau:i\in \mathcal{I}\}|}{|\mathcal{I}|},$$
that is, the probability that model $m$ is within a factor $\tau$ of the best performing model. The performance profiles in Figure \ref{fig:models_ppcombined2} have been plotted on the log-scale for clarity.

The top-left performance profile in Figure \ref{fig:models_ppcombined2} includes data from all instances, whilst the three other performance profiles consider the three sizes of instance separately. We see that for $n=10$, the matching-based model performs slightly better than the assignment-based model, however the inclusion of a warm-start does not seem to improve the matching model. For $n=15$ and $n=20$ however, the assignment model is stronger than the matching model. The benefits of a warm-start solution become most apparent when solving the largest instances, where a warm-start increases both solution times and the number of instances solved to optimality of both the matching and assignment model. 

\begin{figure}[htbp]
	\centering
	\includegraphics[scale=0.5]{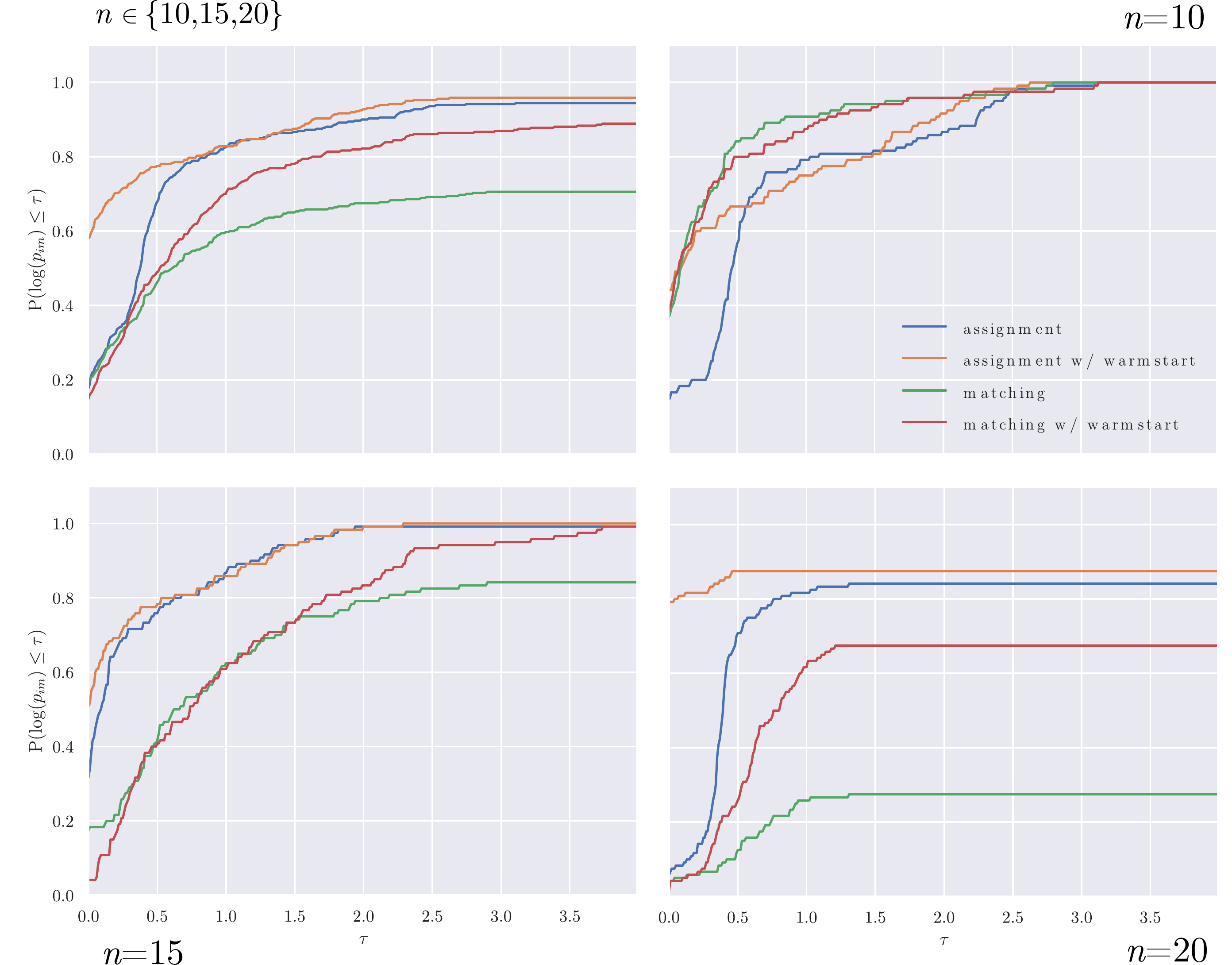}
	\caption{Performance profiles of relative solution times for different instance sizes.}
	\label{fig:models_ppcombined2}
\end{figure}

\subsection{Model parameters}

We now examine the impact of the model parameters $\Gamma$ and $\Delta$ on the objective value. For each set of instances, Tables \ref{table:Gamma_objval} and \ref{table:Delta_objval} report the average objective value of the best known feasible solutions found by any method for different values of $\Gamma$ and $\Delta$ respectively, as well as the relative percentage difference in this average from the sets of instances where $\Gamma=3$ and $\Delta=0$, respectively. 

The results in Table \ref{table:Gamma_objval} show that, as we would expect, increasing the $\Gamma$ increases the average objective value in a concave manner. Table \ref{table:Delta_objval} shows that the inclusion of a second-stage recourse solution provides an improvement in objective value. However we also see that beyond $\Delta=1$, increasing $\Delta$ provides little additional benefit. That is, the vast majority of the benefit of allowing a recourse solution can be captured by allowing just a single swap to the first-stage schedule. However, it is important to note the effect of having been limited to instances sizes of 20 and less by the computational intensity of solving the proposed exact models. We expect that for larger instance sizes, a less restricted and more powerful recourse action, i.e. increasing $\Delta$, would become more advantageous. Additionally, for a discrete budgeted uncertainty set where $p_i\in\{\hat{p}_i,\hat{p}_i+\bar{p}_i\}$ for each $i\in \N$, we might expect the benefits of increasing $\Delta$ to be more apparent, since in this case the adversary is unable to spread the delay across multiple jobs in an attempt to preempt the recourse response, as is currently the case under the continuous budgeted uncertainty set that we consider. The impact of discrete budgeted uncertainty is an interesting possibility for future research on this problem.

\begin{table}[]
\centering
\parbox{.45\linewidth}{
	\centering
\begin{tabular}{lllrr}
\hline \hline
$n$  & $\Gamma$ & $\Delta$ & \textit{avg. best} & $\%$\textit{diff.} \\
\hline
10 & 3     & 2     & 3946.5    & 0.0       \\
10 & 5     & 2     & 4578.0    & 14.1      \\
10 & 7     & 2     & 5053.0    & 22.1      \\
15 & 3     & 2     & 7164.9    & 0.0       \\
15 & 5     & 2     & 8177.5    & 12.7      \\
15 & 7     & 2     & 9002.1    & 20.7      \\
20 & 3     & 2     & 11814.3   & 0.0       \\
20 & 5     & 2     & 13317.8   & 11.5      \\
20 & 7     & 2     & 14582.5   & 19.3      \\ 
\hline\hline
\end{tabular}
\caption{The effects of increasing $\Gamma$ on the average objective value of the best known solution.}
\label{table:Gamma_objval}
}
\hfill
\parbox{.45\linewidth}{
	\centering
\begin{tabular}{lllrr}
\hline \hline
$n$  & $\Gamma$ & $\Delta$ & \textit{avg. best} & $\%$\textit{diff.} \\
\hline
10 & 7     & 0     & 5079.7    & 0.0       \\
10 & 7     & 1     & 5053.0    & -0.5       \\
10 & 7     & 2     & 5053.0    & -0.5       \\
10 & 7     & 3     & 5053.0    & -0.5       \\
15 & 7     & 0     & 9093.8    & 0.0       \\
15 & 7     & 1     & 9002.2    & -1.0       \\
15 & 7     & 2     & 9002.1    & -1.0       \\
15 & 7     & 3     & 9002.1    & -1.0       \\
20 & 7     & 0     & 14735.6   & 0.0       \\
20 & 7     & 1     & 14583.7   & -1.0       \\
20 & 7     & 2     & 14582.5   & -1.0       \\
20 & 7     & 3     & 14582.5   & -1.0       \\ 
\hline\hline
\end{tabular}
\caption{The effects of increasing $\Delta$ on the average objective value of the best known solution.}
\label{table:Delta_objval}
}
\end{table}

\section{Conclusions}\label{section:conclusions}

This paper has introduced a recoverable robust model for the single machine scheduling problem with the total flow time criterion. A general result that allows for the construction of compact formulations for a wide range of recoverable robust problems has been presented, and this approach has been applied to the specific scheduling problem we consider. We have analysed the incremental subproblem of the robust scheduling problem in detail in an attempt to develop more tailored and effective compact formulations for this problem. Specifically, we have proved that matching problems with edge weights of the form of (\ref{obj:thm1}) have integral solutions, and therefore the inclusion of the odd-cycle constraints of the standard matching polytope is unnecessary. This result allows us to derive a matching-based compact formulation for the full recoverable robust single machine scheduling problem. A symmetric assignment-based formulation has also been presented, and we show how the integral matching result can be transferred to this alternative formulation to enable the derivation of a third compact model for this problem. Computational results show that this assignment-based model is the strongest of the three exact models.

There remain a number of promising directions in which future research on this problem can develop. Firstly, in this work we have considered a limited recourse action of allowing $\Delta$ disjoint swaps to be made to the first-stage schedule. Other measures of distance between the first and second-stage solution are certainly possible and worth investigating, especially if the restriction that the swapped pairs be disjoint could be relaxed, and interchanges between the positions of three or more jobs simultaneously can be factored into a recourse action. Another obvious avenue for future research is the analysis of this problem in the context of uncertainty sets different from budgeted uncertainty. Given the vast number of different objective criteria that have been used for single-machine scheduling problems and the unique properties of each, it would be interesting and worthwhile to investigate the application of this recoverable robust model to some of these. As a final suggestion, given the limited size of instance that have been solved by the exact models we propose, an accurate and effective heuristic approach for solving large-scale instances of this problem would certainly be a valuable development.

\appendix

\section*{Acknowledgements}

The authors would like to thank the reviewers for the part they played in the improvement of this paper with their constructive and insightful feedback. The authors are also grateful for the support of the EPSRC-funded (EP/L015692/1) STOR-i Centre for Doctoral Training.

\bibliography{paper}

\begin{thebibliography}{}

\bibitem[Aissi et~al., 2011]{aissi2011minimizing}
Aissi, H., Aloulou, M.~A., and Kovalyov, M.~Y. (2011).
\newblock Minimizing the number of late jobs on a single machine under due date
  uncertainty.
\newblock {\em Journal of Scheduling}, 14(4):351--360.

\bibitem[Aloulou and Della~Croce, 2008]{aloulou2008complexity}
Aloulou, M.~A. and Della~Croce, F. (2008).
\newblock Complexity of single machine scheduling problems under scenario-based
  uncertainty.
\newblock {\em Operations Research Letters}, 36(3):338--342.

\bibitem[Balinski, 1965]{balinski1965integer}
Balinski, M.~L. (1965).
\newblock Integer programming: Methods, uses, computations.
\newblock {\em Management Science}, 12(3):253--313.

\bibitem[Bendotti et~al., 2019]{bendotti2019anchor}
Bendotti, P., Chr{\'e}tienne, P., Fouilhoux, P., and Pass-Lanneau, A. (2019).
\newblock The anchor-robust project scheduling problem.
\newblock {\em HAL preprint hal-02144834, version 1}.

\bibitem[Bertsimas and Sim, 2004]{bertsimas2004price}
Bertsimas, D. and Sim, M. (2004).
\newblock The price of robustness.
\newblock {\em Operations research}, 52(1):35--53.

\bibitem[Bold and Goerigk, 2021a]{bold2020compact}
Bold, M. and Goerigk, M. (2021a).
\newblock A compact reformulation of the two-stage robust resource-constrained
  project scheduling problem.
\newblock {\em Computers \& Operations Research}, 130:105232.

\bibitem[Bold and Goerigk, 2021b]{bold2021recoverable}
Bold, M. and Goerigk, M. (2021b).
\newblock Investigating the recoverable robust single machine scheduling
  problem under interval uncertainty.
\newblock {\em arXiv preprint arXiv:2107.09310}.

\bibitem[Bougeret et~al., 2019]{bougeret2019robust}
Bougeret, M., Pessoa, A.~A., and Poss, M. (2019).
\newblock Robust scheduling with budgeted uncertainty.
\newblock {\em Discrete Applied Mathematics}, 261:93--107.

\bibitem[Bruni et~al., 2017]{bruni2017adjustable}
Bruni, M.~E., Pugliese, L. D.~P., Beraldi, P., and Guerriero, F. (2017).
\newblock An adjustable robust optimization model for the resource-constrained
  project scheduling problem with uncertain activity durations.
\newblock {\em Omega}, 71:66--84.

\bibitem[Bruni et~al., 2018]{bruni2018computational}
Bruni, M.~E., Pugliese, L. D.~P., Beraldi, P., and Guerriero, F. (2018).
\newblock A computational study of exact approaches for the adjustable robust
  resource-constrained project scheduling problem.
\newblock {\em Computers \& Operations Research}, 99:178--190.

\bibitem[Buchheim and Kurtz, 2017]{buchheim2017min}
Buchheim, C. and Kurtz, J. (2017).
\newblock Min--max--min robust combinatorial optimization.
\newblock {\em Mathematical Programming}, 163(1-2):1--23.

\bibitem[Chang et~al., 2017]{chang2017distributionally}
Chang, Z., Song, S., Zhang, Y., Ding, J.-Y., Zhang, R., and Chiong, R. (2017).
\newblock Distributionally robust single machine scheduling with risk aversion.
\newblock {\em European Journal of Operational Research}, 256(1):261--274.

\bibitem[Daniels and Kouvelis, 1995]{daniels1995robust}
Daniels, R.~L. and Kouvelis, P. (1995).
\newblock Robust scheduling to hedge against processing time uncertainty in
  single-stage production.
\newblock {\em Management Science}, 41(2):363--376.

\bibitem[Dolan and Mor{\'e}, 2002]{dolan2002benchmarking}
Dolan, E.~D. and Mor{\'e}, J.~J. (2002).
\newblock Benchmarking optimization software with performance profiles.
\newblock {\em Mathematical programming}, 91(2):201--213.

\bibitem[Edmonds, 1965]{edmonds1965maximum}
Edmonds, J. (1965).
\newblock Maximum matching and a polyhedron with 0,1-vertices.
\newblock {\em Journal of Research of the National Bureau of Standards},
  69B:125--130.

\bibitem[Fischer et~al., 2020]{fischer2020investigation}
Fischer, D., Hartmann, T.~A., Lendl, S., and Woeginger, G.~J. (2020).
\newblock An investigation of the recoverable robust assignment problem.
\newblock {\em arXiv preprint arXiv:2010.11456}.

\bibitem[Fridman et~al., 2020]{fridman2020minimizing}
Fridman, I., Pesch, E., and Shafransky, Y. (2020).
\newblock Minimizing maximum cost for a single machine under uncertainty of
  processing times.
\newblock {\em European Journal of Operational Research}, 286(2):444--457.

\bibitem[Graham et~al., 1979]{graham1979optimization}
Graham, R.~L., Lawler, E.~L., Lenstra, J.~K., and Kan, A.~R. (1979).
\newblock Optimization and approximation in deterministic sequencing and
  scheduling: a survey.
\newblock In {\em Annals of discrete mathematics}, volume~5, pages 287--326.
  Elsevier.

\bibitem[Hanasusanto et~al., 2015]{hanasusanto2015k}
Hanasusanto, G.~A., Kuhn, D., and Wiesemann, W. (2015).
\newblock K-adaptability in two-stage robust binary programming.
\newblock {\em Operations Research}, 63(4):877--891.

\bibitem[Kasperski, 2005]{kasperski2005minimizing}
Kasperski, A. (2005).
\newblock Minimizing maximal regret in the single machine sequencing problem
  with maximum lateness criterion.
\newblock {\em Operations Research Letters}, 33(4):431--436.

\bibitem[Kasperski and Zieli{\'n}ski, 2008]{kasperski20082}
Kasperski, A. and Zieli{\'n}ski, P. (2008).
\newblock A 2-approximation algorithm for interval data minmax regret
  sequencing problems with the total flow time criterion.
\newblock {\em Operations Research Letters}, 36(3):343--344.

\bibitem[Kasperski and Zielinski, 2014]{kasperski2014minmax}
Kasperski, A. and Zielinski, P. (2014).
\newblock Minmax (regret) scheduling problems.
\newblock {\em Sequencing and scheduling with inaccurate data}, pages 159--210.

\bibitem[Kasperski and Zieli{\'n}ski, 2016]{kasperski2016single}
Kasperski, A. and Zieli{\'n}ski, P. (2016).
\newblock Single machine scheduling problems with uncertain parameters and the
  {OWA} criterion.
\newblock {\em Journal of Scheduling}, 19(2):177--190.

\bibitem[Kasperski and Zieli{\'n}ski, 2019]{kasperski2019risk}
Kasperski, A. and Zieli{\'n}ski, P. (2019).
\newblock Risk-averse single machine scheduling: complexity and approximation.
\newblock {\em Journal of Scheduling}, 22(5):567--580.

\bibitem[Kouvelis and Yu, 1997]{kouvelis1997robust}
Kouvelis, P. and Yu, G. (1997).
\newblock {\em Robust discrete optimization and its applications}.
\newblock Kluwer Academic Publishers Dordrecht, Netherlands.

\bibitem[Lebedev and Averbakh, 2006]{lebedev2006complexity}
Lebedev, V. and Averbakh, I. (2006).
\newblock Complexity of minimizing the total flow time with interval data and
  minmax regret criterion.
\newblock {\em Discrete Applied Mathematics}, 154(15):2167--2177.

\bibitem[Liebchen et~al., 2009]{liebchen2009concept}
Liebchen, C., L{\"u}bbecke, M., M{\"o}hring, R., and Stiller, S. (2009).
\newblock The concept of recoverable robustness, linear programming recovery,
  and railway applications.
\newblock In {\em Robust and online large-scale optimization}, pages 1--27.
  Springer.

\bibitem[Lu et~al., 2012]{lu2012robust}
Lu, C.-C., Lin, S.-W., and Ying, K.-C. (2012).
\newblock Robust scheduling on a single machine to minimize total flow time.
\newblock {\em Computers \& Operations Research}, 39(7):1682--1691.

\bibitem[Lu et~al., 2014]{lu2014robust}
Lu, C.-C., Ying, K.-C., and Lin, S.-W. (2014).
\newblock Robust single machine scheduling for minimizing total flow time in
  the presence of uncertain processing times.
\newblock {\em Computers \& Industrial Engineering}, 74:102--110.

\bibitem[Mastrolilli et~al., 2013]{mastrolilli2013single}
Mastrolilli, M., Mutsanas, N., and Svensson, O. (2013).
\newblock Single machine scheduling with scenarios.
\newblock {\em Theoretical Computer Science}, 477:57--66.

\bibitem[Montemanni, 2007]{montemanni2007mixed}
Montemanni, R. (2007).
\newblock A mixed integer programming formulation for the total flow time
  single machine robust scheduling problem with interval data.
\newblock {\em Journal of Mathematical Modelling and Algorithms},
  6(2):287--296.

\bibitem[Schrijver, 2003]{schrijver2003combinatorial}
Schrijver, A. (2003).
\newblock {\em Combinatorial optimization: polyhedra and efficiency},
  volume~24.
\newblock Springer Science \& Business Media.

\bibitem[Tadayon and Smith, 2015]{tadayon2015algorithms}
Tadayon, B. and Smith, J.~C. (2015).
\newblock Algorithms and complexity analysis for robust single-machine
  scheduling problems.
\newblock {\em Journal of Scheduling}, 18(6):575--592.

\bibitem[Thomas, 2015]{thomas2016matching}
Thomas, D.~J. (2015).
\newblock {\em Matching Problems with Additional Resource Constraints}.
\newblock PhD thesis, Universit\"at Trier.

\bibitem[Yang and Yu, 2002]{yang2002robust}
Yang, J. and Yu, G. (2002).
\newblock On the robust single machine scheduling problem.
\newblock {\em Journal of Combinatorial Optimization}, 6(1):17--33.

\bibitem[Zhao et~al., 2010]{zhao2010family}
Zhao, H., Zhao, M., et~al. (2010).
\newblock A family of inequalities valid for the robust single machine
  scheduling polyhedron.
\newblock {\em Computers \& Operations Research}, 37(9):1610--1614.

\end{thebibliography}

\section{Omitted proofs}\label{section:proofs}

\begin{proof}[Proof of Theorem~\ref{theoremlp}]
	Let an instance $I$ of problem (\ref{assign1})-(\ref{assign3}) be given, and define 
	\[ P = \{ \bm{y}\in \mathbb{R}_+^{n\times n} : \eqref{inc2}-\eqref{inc5} \}. \]
	To solve $I$, we construct a graph with a single node for each job $i\in \N$, where an edge between nodes $i$ and $j$ indicates that jobs $i$ and $j$ swap their positions from the first-stage schedule. Since edges correspond to unique swaps, the set of edges in this graph is given by $\cE=\{(i,j):i,j\in \N,\, j>i\}$. The weight of an edge $(i,j)$ in this graph is equal to the reduction in objective cost from making the corresponding swap, i.e. $a_ib_i+a_jb_j-a_ib_j-a_jb_i=(a_i-a_j)(b_i-b_j)$. We aim to choose up to $\Delta$ edges from this graph to maximise the reduction in objective cost. That is, given $I$, we construct an instance $J$ of the following cardinality-constrained matching problem:`
\begin{align}
	\min\ &\sum_{i\in \N}a_ib_i-\sum_{(i,j)\in \cE}(a_i-a_j)(b_i-b_j)z_{ij}\\
	\text{s.t. } & \sum_{(i,j)\in \cE} z_{ij} + \sum_{(j,i)\in \cE}z_{ij} \leq 1 & \forall i \in \N\label{matchingconstr1}\\
	      & \sum_{(i,j)\in \cE} z_{ij}\leq \Delta\\
	      & z_{ij} \geq 0 & \forall (i,j) \in \cE\label{matchingconstr3}.
\end{align}
As stated in Corollary \ref{cor:ccmatching}, this problem has an integral optimal solution. Hence, letting $P'=\{ \bm{z} \in \mathbb{R}_{+}^{|\cE|} : \eqref{matchingconstr1}-\eqref{matchingconstr3}\}$, we construct mappings $\phi:P\rightarrow P^{'}$ and $\phi^{-1}:P^{'}\rightarrow P$ which preserve objective value and integrality, showing problems $I$ and $J$ are indeed equivalent. To this end, we define $\phi(\bm{y})=\bm{z}$ by $z_{ij}=y_{ij}$ for all $(i,j)\in \cE$. Observe that
\begin{align*}
obj_{I} (\bm{y}) &= \sum_{i\in \N}\sum_{j\in \N} a_ib_jy_{ij}\\
&=\sum_{i\in \N}a_ib_iy_{ii} + \sum_{i\in \N}\sum_{j\in \N:j\neq i}a_ib_jy_{ij}\\
&=\sum_{i\in \N}a_ib_iy_{ii} + \sum_{i\in \N}\sum_{j\in \N:j>i}(a_ib_j+a_jb_i)y_{ij}\\
&=\sum_{i\in \N}a_ib_i\Bigg(1-\sum_{j\in \N:j\neq i}y_{ij}\Bigg) + \sum_{i\in \N}\sum_{j\in \N:j>i}(a_ib_j+a_jb_i)y_{ij}\\
&=\sum_{i\in \N}a_ib_i - \sum_{i\in \N}\sum_{j\in \N:j\neq i}a_ib_iy_{ij} + \sum_{i\in \N}\sum_{j\in \N:j>i}(a_ib_j+a_jb_i)y_{ij}\\
&=\sum_{i\in \N}a_ib_i - \sum_{i\in \N}\sum_{j\in \N:j> i}(a_ib_i+a_jb_j)y_{ij} + \sum_{i\in \N}\sum_{j\in \N:j>i}(a_ib_j+a_jb_i)y_{ij}\\
&=\sum_{i\in \N}a_ib_i - \sum_{i\in \N}\sum_{j\in \N:j> i}(a_ib_i+a_jb_j-a_ib_j-a_jb_i)y_{ij}\\
&=\sum_{i\in \N}a_ib_i - \sum_{(i,j)\in \cE}(a_i-a_j)(b_i-b_j)z_{ij}\\
&=obj_{J}(\bm{z}).
\end{align*}
Conversely, we define $\phi^{-1}(\bm{z})=\bm{y}$ with $y_{ij}=y_{ji}=z_{ij}$ for each $(i,j)\in \cE$, and $y_{ii}=(1-\sum_{j\in \N:j>i}z_{ij})(1-\sum_{j\in \N:j<i}z_{ji})$ for each $i\in \N$. Then 

\begin{align*}
	obj_{J}(\bm{z}) &=\sum_{i\in \N}a_ib_i - \sum_{(i,j)\in \cE}(a_i-a_j)(b_i-b_j)z_{ij}\\
			&=\sum_{i\in \N}a_ib_i-\sum_{i\in \N}\sum_{j\in \N:j>i}(a_ib_i+a_jb_j)z_{ij}+\sum_{i\in \N}\sum_{j\in \N:j>i}(a_ib_j+a_jb_i)z_{ij}\\
			&=\sum_{i\in \N}a_ib_i-\sum_{i\in \N}a_ib_i\sum_{j\in \N:j>i}z_{ij}-\sum_{i\in \N}a_ib_i\sum_{j\in \N:j<i}z_{ji}\\
			&&\hspace{-9cm}	+\sum_{i\in \N}a_ib_i\underbrace{\sum_{j\in \N:j>i}z_{ij}\sum_{j\in \N:j<i}z_{ji}}_\text{$=0$}+\sum_{i\in \N}\sum_{j\in \N:j>i}(a_ib_j+a_jb_i)z_{ij}\\
			&=\sum_{i\in \N}a_ib_i\Bigg(1-\sum_{j\in \N:j>i}z_{ij}\Bigg)\Bigg(1-\sum_{j\in \N:j<i}z_{ji}\Bigg)+\sum_{i\in \N}\sum_{j\in \N:j>i}(a_ib_j+a_jb_i)z_{ij}\\
			&=\sum_{i\in \N}a_ib_iy_{ii}+\sum_{i\in \N}\sum_{j\in \N:j>i}a_ib_jy_{ij}+\sum_{i\in \N}\sum_{j\in \N:j>i}a_jb_iy_{ji}\\
			&=\sum_{i\in \N}a_ib_iy_{ii}+\sum_{i\in \N}\sum_{j\in \N:j\neq i}a_ib_jy_{ij}\\
			&=\sum_{i\in \N}a_ib_jy_{ij}\\
			&=obj_{I}(\bm{y}).
\end{align*}
Therefore $\min_{\bm{y}\in P} obj_{I}(\bm{y}) = \min_{\bm{z}\in P'} obj_J(\bm{z})$. Since problem (\ref{matchinc1})-(\ref{matchinc4}) has an optimal integral solution, there must also be an optimal integral solution to $I$ via the mapping $\phi^{-1}$, proving the claim.
\end{proof}

\begin{proof}[Proof of Theorem~\ref{prop:dominance}]
To prove this, it will suffice to show that $P(F_m) \subsetneq P(F_a)$, where $P(F_m)$ and $P(F_a)$ denote the sets of feasible solutions to the linear relaxations of the non-linear matching-based and assignment-based formulations respectively, i.e. that for each $(\bm{x},\, \bm{z},\, \bm{q})\in P(F_m)$, there exists a $\bm{y}$ such that $(\bm{x},\, \bm{y},\, \bm{q})\in P(F_a)$. Note that the same vector $\bm{q}$ is used in both formulations, which results in the objective values being the same. To this end, let $\bm{z}$ be part of a feasible solution to the matching formulation, and define the transformation $\phi(\bm{z}) = \bm{y}$ by $y_{ij}=y_{ji}=z_{ij}$ for $(i,j)\in \cE=\{(i,j)\in \mathcal{N}\times\mathcal{N}:i<j\}$ and $y_{ii}=1-\sum_{j\in \N:j>i}z_{ij}-\sum_{j\in \N:j<i}z_{ji}$ for $i\in \N$.

Firstly, observe that the assignment constraints (\ref{assignment_model:assignment_constr1}) and (\ref{assignment_model:assignment_constr2}) are satisfied by this definition of $\bm{y}$. For each $j\in \N$ we have that
\begin{align*}
\sum_{i\in \N}y_{ij}&=\sum_{i\in \N:i<j}y_{ij} + \sum_{i\in \N:i>j}y_{ji} + y_{ii}\\
&=\sum_{i\in \N:i<j}z_{ij} + \sum_{i\in \N:i>j}z_{ji} + 1-\sum_{i\in \N:i<j}z_{ij}-\sum_{i\in \N:i>j}z_{ji} = 1. 
\end{align*}
The same can be shown for the `incoming' assignment constraints for each $i\in \N$.

Now observe that constraint (\ref{matching_model:constr3}) states that 
$$\sum_{i\in \N}\sum_{j\in \N:j>i}z_{ij}\leq \Delta,$$
or equivalently
$$\sum_{i\in \N}\sum_{j\in \N:j>i}y_{ij}\leq \Delta.$$
Since, $y_{ij}=y_{ji}$ for all $i,\,j\in \N$, we have that
$$\sum_{i\in \N}\bigg(\sum_{j\in \N:j>i}y_{ij}+\sum_{j\in \N:j<i}y_{ij}\bigg)\leq 2\Delta,$$
and therefore as a consequence of the assignment constraints, which imply
$$\sum_{i\in \N}\sum_{j\in \N}y_{ij} = \sum_{i\in \N}\bigg(\sum_{j\in \N:j>i}y_{ij}+\sum_{j\in \N:j<i}y_{ij}+y_{ii}\bigg)=n,$$
we have that
$$\sum_{i\in \N}y_{ii}=n - \sum_{i\in \N}\bigg(\sum_{j\in \N:j>i}y_{ij}+\sum_{j\in \N:j<i}y_{ij}\bigg) \geq n-2\Delta.$$
This is exactly constraint (\ref{assignment_model:constr1}) from the assignment-based formulation.

Finally, consider constraints (\ref{matching_model:constr1}):
\begin{align*}
\sum_{m\in\cM} a_{mi}q_m +&\sum_{j\in \N:j>i}\Bigg(\sum_{\ell\in \N}\ell\cdot x_{j\ell}-\sum_{\ell\in \N}\ell\cdot x_{i\ell}\Bigg)z_{ij}\\
&&\hspace{-60mm}-\sum_{j\in \N:j<i}\Bigg(\sum_{\ell\in \N}\ell \cdot x_{i\ell}-\sum_{\ell\in \N}\ell\cdot x_{j\ell}\Bigg)z_{ji}\geq (n+1-\sum_{\ell \in \N}\ell\cdot x_{i\ell})\quad\forall i\in \N.
\end{align*}
These can be rewritten as
\begin{align*}
\sum_{m\in\cM} a_{mi}q_{m} &\geq (n+1-\sum_{\ell\in \N}\ell\cdot x_{i\ell})+\sum_{j\in \N: j<i}z_{ji}\sum_{\ell\in\N}\ell\cdot x_{i\ell}+\sum_{j\in \N: j>i}z_{ij}\sum_{\ell\in\N}\ell\cdot x_{i\ell}\\
&\hspace{4cm}-\sum_{j\in \N:j<i}z_{ji}\sum_{\ell\in \N}\ell\cdot x_{j\ell}-\sum_{j\in \N:j>i}z_{ij}\sum_{\ell\in \N}\ell\cdot x_{j\ell} \\
&= (n+1-\sum_{\ell\in \N}\ell\cdot x_{i\ell})+\sum_{j\in \N: j<i}y_{ij}\sum_{\ell\in\N}\ell\cdot x_{i\ell}+\sum_{j\in \N: j>i}y_{ij}\sum_{\ell\in\N}\ell\cdot x_{i\ell}\\
&\hspace{4cm}-\sum_{j\in \N:j<i}y_{ij}\sum_{\ell\in \N}\ell\cdot x_{j\ell}-\sum_{j\in \N:j>i}y_{ij}\sum_{\ell\in \N}\ell\cdot x_{j\ell} \\
&=  (n+1-\sum_{\ell\in \N}\ell\cdot x_{i\ell})+\sum_{j\in \N: j\neq i}y_{ij}\sum_{\ell\in\N}\ell\cdot x_{i\ell}-\sum_{j\in \N:j\neq i}y_{ij}\sum_{\ell\in \N}\ell\cdot x_{j\ell} \\
&=  n+1-\bigg(1-\sum_{j\in \N: j\neq i}y_{ij}\bigg)\sum_{\ell\in\N}\ell\cdot x_{i\ell}-\sum_{j\in \N:j\neq i}y_{ij}\sum_{\ell\in \N}\ell\cdot x_{j\ell}\\
&= n+1 - y_{ii}\sum_{\ell\in\N}\ell\cdot x_{i\ell}-\sum_{j\in \N:j\neq i}y_{ij}\sum_{\ell\in \N}\ell\cdot x_{j\ell}\\
&= n+1 - \sum_{j\in \N}y_{ij}\sum_{\ell\in \N}\ell\cdot x_{j\ell}\\
&= (n+1)\sum_{j\in\N} y_{ij} - \sum_{j\in \N}y_{ij}\sum_{\ell\in \N}\ell\cdot x_{j\ell}\\
& = \sum_{j\in \N}\big(n+1-\sum_{\ell\in\N}\ell\cdot x_{j\ell}\big)y_{ij},
\end{align*}
for each $i\in \N$ which are constraints (\ref{assignment_model:constr3}) from the assignment formulation.

Hence, every feasible solution to the non-linear matching-based formulation has a corresponding feasible solution for the non-linear assignment-based formulation, i.e. $P(F_m) \subseteq P(F_a)$. The examples used in the proof of Theorem \ref{thm:incomparable_formulations} show that there exist instances for which the LP bound of the matching formulation is strictly larger than that of the assignment formulation, demonstrating that $P(F_m) \subset P(F_a)$.
\end{proof}

\section{Complete formulations}\label{section:appendix}

\subsection{General model}\label{sec:model1}

The complete linear compact formulation for the general recoverable robust model presented in Section \ref{sec:general} is as follows:
\begin{align}
\min\ & \sum_{m\in\cM} b_mq_m \label{generalfull:first} \\
\text{s.t. } & \sum_{k\in\cK} \mu_k = 1 \\
	     & \sum_{m\in\cM} a_{mi} q_m \ge \sum_{k\in\cK} \left(\sum_{j\in \N}(n+1-j)\sum_{i'\in \N}h^k_{ii'j}\right) & \forall i\in \N \\
& \sum_{i'\in \N} z^k_{ii'} = 1 & \forall i\in \N, \, k\in\cK \\
& \sum_{i\in \N} z^k_{ii'} = 1 & \forall i'\in \N, \, k\in\cK  \\
& z^k_{ii'} = z^k_{i'i} & \forall i,i'\in \N, \, k\in\cK  \\
& \sum_{i\in \N} z^k_{ii} \ge n - 2\Delta & \forall k\in\cK\\
& \sum_{j\in \N} x_{ij} = 1 & \forall i\in \N \\
& \sum_{i\in \N} x_{ij} = 1 & \forall j\in \N \\ 
& w^k_{ii'j} \leq z^k_{ii'} & \forall i,i',j\in \N, \, k\in\cK\\
& w^k_{ii'j} \leq x_{i'j} & \forall i,i',j\in \N, \, k\in\cK\\
& w^k_{ii'j} \geq z^k_{ii'} + x_{i'j} -1 & \forall i,i',j\in \N, \, k\in\cK\\
& h^k_{ii'j} \leq w^k_{ii'j} & \forall i,i',j\in \N, \, k\in\cK\\
& h^k_{ii'j} \leq \mu_k & \forall i,i',j\in \N, \, k\in\cK\\
& h^k_{ii'j} \geq \mu_k + w^k_{ii'j} - 1 & \forall i,i',j\in \N, \, k\in\cK\\
& w^k_{ii'j} \in \{0,1\} &\forall i,i',j\in \N, \, k\in\cK\\
& h^k_{ii'j} \geq 0 & \forall i,i',j\in \N, \, k\in\cK\\
& \mu_k \geq 0 & \forall k\in \cK\\
& q_m \geq 0 &\forall m\in\cM\\
& z_{ii'}^k\in \{0,1\} & \forall i,\,i' \in \N,\,k\in\cK\\
& x_{ij}\in \{0,1\}& \forall i,\,j\in \N. \label{generalfull:last}
\end{align}

\subsection{Matching-based model}\label{sec:model2}

The fully-linearised formulation of the matching-based model presented in Section \ref{subsection:matching_formulation} is as follows:
\begin{align}
	\min_{\bm{x},\,\bm{z},\,\bm{u},\,\bm{v},\, \bm{q}}\ &\sum_{m\in\cM} b_mq_m\label{matchingfull:first}\\
	\text{s.t. }&\sum_{(i,j)\in \cE}z_{ij}+\sum_{(j,i)\in \cE}z_{ji}\leq 1 & \forall i\in \N \\
		    &\sum_{(i,j)\in \cE}z_{ij}\leq \Delta\label{matchingfull:delta_constr}\\
		    &\sum_{m\in\cM} a_{mi}q_m+\sum_{(i,j)\in \cE}\Bigg(\sum_{\ell\in \N}\ell\cdot v_{ij\ell}-\sum_{\ell\in \N}\ell\cdot u_{ij\ell}\Bigg) \nonumber\\
		    &\hspace{0.5cm}-\sum_{(j,i)\in \cE}\Bigg(\sum_{\ell\in \N}\ell\cdot v_{ji\ell}-\sum_{\ell\in \N}\ell\cdot u_{ji\ell}\Bigg) \geq (n+1-\sum_{\ell \in \N}\ell\cdot x_{i\ell})&\hspace{-2cm}\forall i\in \N\label{matchingfull:constr}\\
		    &\sum_{i\in \N}x_{i\ell} = 1&\forall \ell \in \N\\
		    &\sum_{\ell\in \N}x_{i\ell} = 1&\forall i \in \N\\
	&u_{ij\ell}\leq x_{i\ell} &\hspace{-5cm}\forall (i,j)\in \cE,\,\ell\in \N\label{matchingfull:u1}\\
	&u_{ij\ell}\leq z_{ij} &\hspace{-5cm}\forall (i,j)\in \cE,\,\ell\in \N\\
	&u_{ij\ell}\geq z_{ij} + x_{i\ell} - 1 &\hspace{-5cm}\forall (i,j)\in \cE,\,\ell\in \N\label{matchingfull:u3}\\
	&v_{ij\ell}\leq x_{j\ell} &\hspace{-5cm}\forall (i,j)\in \cE,\,\ell\in \N\label{matchingfull:v1}\\
	&v_{ij\ell}\leq z_{ij} &\hspace{-5cm}\forall (i,j)\in \cE,\,\ell\in \N\\
	&v_{ij\ell}\geq z_{ij} + x_{j\ell} -1 &\hspace{-5cm}\forall (i,j)\in \cE,\,\ell\in \N\label{matchingfull:v3}\\
	&u_{ij\ell}\geq0 &\hspace{-5cm}\forall (i,j)\in \cE,\,\ell \in \N\\
	&v_{ij\ell}\geq0 &\hspace{-5cm}\forall (i,j)\in \cE,\,\ell \in \N\\
	&q_m\ge 0 & \hspace{-5cm} \forall m\in\cM\\
	&z_{ij}\geq 0&\hspace{-5cm}\forall (i,j)\in \cE\\
	&x_{i\ell}\in\{0,1\}&\hspace{-5cm}\forall i,\ell\in \N\label{matchingfull:last}.
\end{align}

\subsection{Assignment-based model}\label{sec:model3}
The fully-linearised formulation of the assignment-based model derived in Section \ref{subsection:assignment_formulation} is as follows:

\begin{align}
	\min_{\bm{x},\,\bm{y},\,\bm{w},\,\bm{q}}\ & \sum_{m\in\cM} b_mq_m \label{assignmentfull:first}\\
	\text{s.t. } & \sum_{i\in \N} y_{ij} = 1 & \forall j\in \N\label{assignmentfull:assignment1}\\
		     & \sum_{j\in \N} y_{ij} = 1 & \forall i\in \N\label{assignmentfull:assignment2}\\
		     & \sum_{i\in \N} y_{ii} \ge n-2\Delta \label{assignmentfull:delta_constr}\\
& y_{ij} = y_{ji} & \forall i,j\in \N \\
& \sum_{m\in\cM}a_{mi}q_m \ge \sum_{j\in \N} \Big((n+1)y_{ij}-\sum_{\ell\in \N}\ell\cdot w_{ij\ell}\Big) & \forall i\in \N\label{assignmentfull:constr}\\
& \sum_{i\in \N} x_{i\ell} = 1 & \forall \ell\in \N\\
& \sum_{\ell\in \N} x_{i\ell} = 1 & \forall i\in \N\\
& w_{ij\ell}\leq x_{j\ell} & \forall i,\,j,\,\ell \in \N\label{assignmentfull:w1}\\
& w_{ij\ell}\leq y_{ij} & \forall i,\,j,\,\ell \in \N\\
& w_{ij\ell}\geq x_{j\ell}+y_{ij}-1 & \forall i,\,j,\,\ell \in \N\label{assignmentfull:w3}\\
& w_{ij\ell}\geq 0 & \forall i,\,j,\,\ell \in \N\\
& q_m\geq 0 & \forall m\in\cM \\
& y_{ij} \ge 0 & \forall i,\,j\in \N\\
& x_{i\ell} \in \{0,1\} & \forall i,\,\ell\in \N.\label{assignmentfull:last}
\end{align}

\end{document}